\documentclass[a4paper,10pt,reqno]{amsart}
\usepackage{amsfonts,amsmath,amssymb,amsthm,amsaddr,datetime}
\usepackage{graphicx}
\usepackage{fixmath}
\usepackage{enumitem}
\usepackage{hyperref}
\usepackage{tikz}
\usepackage{mathrsfs}
\usepackage{mathtools}
\usepackage{stmaryrd}
\usepackage{psfrag}
\usepackage{epsf}
\usepackage{url,rotating}
\usepackage[cp1250]{inputenc}
\usepackage{amsmath,amsfonts,amssymb,latexsym}
\usepackage{setspace}
\usepackage{enumitem}
\usepackage{multirow}
\usepackage[margin=1.2in]{geometry}
\usepackage{makecell} %for \Gape to space rows in table

\newtheorem{thm}{Theorem}
\newtheorem{conj}{Conjecture}

\newtheorem{lemma}{Lemma}[section]

\newtheorem{claim}{Claim}[section]

\newenvironment{subproof}{\noindent \emph{Subproof.}}{\hfill $\Box $}
\newenvironment{kst}
{\setlength{\leftmargini}{2\parindent}
 \begin{itemize}
 \setlength{\itemsep}{-1.1mm}}
{\end{itemize}}

\begin{document}

     %\pagecolor{cyan}

\pagestyle{plain}
\setlength{\baselineskip}{15pt}
\title{Bipartite graphs with minimum degree at least 15 are antimagic}

\author{Kecai Deng}
\address{
School of Mathematical Sciences, Huaqiao University, \\
              Quanzhou 362000, Fujian, P.R. China
}

\email{\qquad kecaideng@126.com}
\thanks{Kecai Deng is the corresponding author. \\
\hspace*{4.3mm}Kecai Deng is partially supported by NSFC (No. 11701195),   Fundamental Research Funds for the Central
Universities (No. ZQN-904) and Fujian Province University Key Laboratory of Computational Science,
School of Mathematical Sciences, Huaqiao University, Quanzhou, China (No. 2019H0015).}

\makeatletter
\@namedef{subjclassname@2020}{\textup{2020} Mathematics Subject Classification}
\makeatother

\begin{abstract}  An antimagic     {labeling}  of a graph $G=(V,E)$   is a one-to-one mapping  $f: E\rightarrow\{1,2,\ldots,|E|\}$, ensuring that the vertex sums in $V$ are pairwise distinct, where a vertex sum of a vertex $v$  is   defined as the sum of the labels of the edges  incident to $v$.  A graph    is called antimagic if it admits an antimagic labeling.
The Antimagic Labeling  Conjecture, proposed by Hartsfield and Ringel in 1990, posits   that  every connected graph other than $K_2$  is antimagic.    The   conjecture was confirmed for graphs of average degree at least 4,182 in 2016 by Eccles, where it was stated that  a similar approach could not reduce the bound below 1,000 from   4,182.

  This paper  shows  that every bipartite graph with minimum degree at least 15 is antimagic. Our approach relies on three   tools:     a consequence of K\"{o}nig's Theorem, the existence of a subgraph of a specific size that avoids Eulerian components, and   a labeling lemma that ensures some vertex sums are divisible by three while others are not.
\end{abstract}

\maketitle

\section{Introduction}
All graphs considered in this paper are simple, undirected, and finite. We follow \cite{Diestel} for the notations and terminology not
defined here.  Let $G=(V,E)$   be a   graph  of size $m$, and let $L$ be an integer set of the same size.   A one-to-one mapping  $f: E\rightarrow L$  is called  a \emph{labeling} on  $G$ with $L$. The \emph{vertex sum} $\sigma_f(v)$ (or $\sigma(v)$ if $f$ is clear)  of a vertex  $v$  is then defined as the   sum of the labels of      the edges incident to $v$.  If $L=\{1,2,\ldots,m\}$ and all vertex sums in $V$ under $f$  are pairwise distinct, then $f$ is called an \emph{antimagic labeling} of $G$. We call $G$ \emph{antimagic} if it permits an antimagic labeling. This concept was introduced by Hartsfield and Ringel in \cite{ref1}, 1990. Their conjecture, stated as follows, has attracted considerable interest in the field:

\begin{conj}{\rm(Antimagic Labeling Conjecture \cite{ref1}\rm)}
	Every connected graph other than $K_{2}$ is antimagic.	
\end{conj}

	 The conjecture remains open, though many families of graphs were shown antimagic. In 2004, Alon    et al. \cite{ref2} showed that    there exists a constant $c$ such that each graph $G$ of order $n$ with   minimum degree at least $c \log n$ is antimagic; they also showed that     each graph $G$ of order $n$  with   maximum degree   at least $n-2$ is antimagic. The first result of Alon    et al.   was improved by  Eccles \cite{Eccles2016} in 2016, where    graphs with average degree
at least  4,182  were shown  antimagic. As noted in \cite{Eccles2016}, it is unlikely that a similar approach could reduce the bound below 1,000 from    4,182. Yilma \cite{ref3} further refined the second result of Alon et al.,  by showing that each graph $G$ of order $n$ with   maximum degree   at least $n-3$ is  antimagic. Other  known families of antimagic graphs   include    regular graphs   \cite{ref17,ref16,ref15}, some biregular bipartite graphs \cite{Deng2022,Yu2023},    each graph    of order $p^k$ ($p$ is odd prime and $k$ is a positive integer) that includes  a $C_p$-factor \cite{Hefetz2005,Hefetz2010}, each forest that includes at most one vertex of degree two \cite{ref6,ref7,Sierra2023}, unions  of graphs with many three-paths \cite{Chavez2023},   	caterpillars \cite{ref4,ref5}, subdivided caterpillars \cite{WDZ},  and so on. For further information on the topic of antimagic labeling of graphs, one is referred to the comprehensive dynamic survey provided in  \cite{ref18}.

This paper demonstrates that every bipartite graph with   minimum degree of at least 15 is antimagic.

  \begin{thm}  \label{thm1} Every bipartite graph with minimum degree at least 15 is antimagic.\end{thm}
Our approach  relies on three tools:     a consequence (Lemma \ref{lem1} in Section 2) of K\"{o}nig's Theorem, the existence of a subgraph of a specific size that avoids Eulerian components (Claim \ref{claim2} in Section 3), and   a labeling lemma (Lemma \ref{lem6}  in Section 2) that ensures some vertex sums are divisible by three while others are not.

The remainder of this paper is organized as follows. Section 2 presents all necessary lemmas for proving Theorem \ref{thm1}, among which Lemma \ref{lem6} is particularly noteworthy for its potential applications beyond the scope of this work.  In Section 3, we develop a constructive framework for a given bipartite graph $G$ with minimum degree $\delta(G)\geq15$, establishing (1) a vertex partition, (2) an edge decomposition, and (3) an integer partition of $\{1,2,\ldots,|E|\}$ that collectively enable the construction of an antimagic labeling for $G$.

\section{Terminology and  lemmas}
  Let   $G=(V,G)$ be a graph, $v\in V$, $V_1,V_2\subseteq V$ and $E_1\subseteq E$.   If each component of $G$ is Eulerian, then $G$ is called an \emph{even graph}. Denoted by $\delta_G(V_1)$ and  $\Delta_G(V_1)$   the minimum and maximum degree in $V_1$, respectively.  Denote the set of edges   incident to $v$ in $E_1$ by $E_1(v)$ and   let  $E_1(V_1)\triangleq\cup_{v\in V_1}E_1(v)$. We say
  $E_1$ \emph{covers} $v$ if $E_1(v)\neq \emptyset$,  and $E_1$ \emph{covers} $V_1$ if  $E_1$ covers each vertex in $V_1$.   If $E_1$ is a matching that  covers $v$ (resp. $V_1$), then we also say $E_1$ \emph{saturates} $v$ (resp. $V_1$). If $V_1\cap V_2=\emptyset$, then denote by $G[V_1,V_2]$ the bipartite graph with vertex set $V_1\cup V_2$  and edge set consisting of    all the edges between $V_1$ and $V_2$. A   \emph{$(v_1,v_{t+1})$-trail} $v_1v_2\ldots v_{t}v_{t+1}$ of $G$ is an ordered   sequence of edges  $e_1,e_2,\ldots, e_{t}$ in $G$, where $e_i=v_iv_{i+1}$ for $i\in[1,t]$  and $e_i\neq e_j$ for distinct $i,j\in [1,t]$.  A $(v_1,v_{t+1})$-trail is called   \emph{open}    if
  $v_1\neq v_{t+1}$.
  In a bipartite graph $G=(X\cup Y, E)$, an open  trail is called a \emph{$U$-trail} if both ends are in $U$ for $U=X,Y$, and  is called an \emph{$XY$-trail} if the two ends are in $X$ and $Y$, respectively. An $X$-trail with two edges is also called an     \emph{$X$-link}. See Figure \ref{fig-5} for example.

 \begin{figure}[htbp]
 \centering
 %\psfrag{a}{(a)}
 %\psfrag{b}{(b)}
 \includegraphics[height=4.6cm]{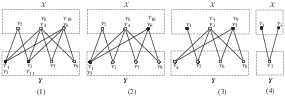}
\caption{\label{fig-5} (1) A  $Y$-trail $v_1v_2\ldots v_{10}v_{11}$;  (2) An $XY$-trail $v_1v_2\ldots v_{9}v_{10}$; (3) An  $X$-trail $v_1v_2\ldots v_{8}v_{9}$; (4) An $X$-link. Filled vertices are the ends of the trails.}
\end{figure}

  For integers $a, b$ with $a\leq b$, denote by
 $[a,b]=\{a,\ldots,b\}$ and by $[a]=[1,a]$ when $a\geq1$. We set $[a,b]=\emptyset$ if $a>b$. For a labeling $f$ on $G$, the \emph{partial vertex sum} $\sigma_f^{E_1}(v)$ denotes the sum of labels of all edges in $E_1$ that are incident to $v$. The notation $\sigma_f^{E_1}(V_1) = \{\sigma_f^{E_1}(v) \mid v \in V_1\}$ represents the multiset of partial vertex sums over $V_1$, while $\sigma_f(V_1) = \{\sigma_f(v) \mid v \in V_1\}$ gives the multiset of complete vertex sums. An integer $j$ qualifies as a \emph{$\mu$-label} precisely when it satisfies the congruence $j \equiv \mu \pmod{3}$ for $\mu \in \{0,1,2\}$. Let  $K\subseteq \{0,1,2\}$. Then each      $\mu$-label is called     a  \emph{$K$-label} for each $\mu\in K$.       A    \emph{$\mu$-sum, $\mu$-vertex sum,  $\mu$-partial vertex sum, $K$-sum, $K$-vertex sum} and  \emph{$K$-partial vertex sum}  are defined analogously.

The following two lemmas are known results: the first follows directly from K\"{o}nig's Theorem \cite{Konig} as presented in \cite{shan}, while the second appears in \cite{ref15}.

\begin{lemma} \label{lem1} \cite{shan} For any bipartite graph $G = (A \cup B, E)$, the vertex set $A \cup B$ admits a partition $X \cup Y$ where:
\begin{enumerate}
    \item $Y$ is an independent set in $G$, and
    \item the bipartite subgraph $G[X,Y]$ contains a matching $M$ saturating $X$.
\end{enumerate}\end{lemma}

 \begin{lemma} \label{lem2} \cite{ref15}   Let $G=(V,E)$ be a   connected graph and   $V_{odd}$ be the set of odd-degree  vertices  in $G$. If $V_{odd}\neq \emptyset$, then $E$   decomposes into exactly
$\frac{|V_{odd}|}{2}$ open trails. \end{lemma}

The following extension of Lemma 2.2 to general graphs  is direct.

 \begin{lemma} \label{lem22}    Let $G=(V,E)$ be a graph without isolated vertices or Eulerian components, and $V_{odd}$ be the set of odd-degree vertices    in $G$. Then $E$   decomposes into exactly
$\frac{|V_{odd}|}{2}$ open trails. \end{lemma}

 The  open   trail decomposition  of $E_G$   in Lemma \ref{lem22}  is called   \emph{good}.

\begin{lemma} \label{lem5} Let $G=(X\cup Y, E)$ be a connected bipartite graph with  at least one odd-degree vertex in each of $X$ and $Y$. Then $E$ admits a good open trail decomposition containing at least one  {$XY$}-{trail}.
\end{lemma}
\begin{proof} By Lemma \ref{lem22},   $E$ admits a good open trail decomposition $\mathcal T$. If there are exactly two
vertices of odd degree, the result follows directly. So assume there exist at least four odd-degree vertices.   We choose $\mathcal T$ such that it includes the most {$XY$}-{trails}.

Suppose to the contrary that $\mathcal T$ admits no  {$XY$}-{trail}. Since  $G$ is connected and $G$ admits  at least one odd-degree vertex in each of $X$ and $Y$,  $\mathcal T$ includes an $(x_{1},x_{2})$-trail $T_1$ and a $(y_{1},y_{2})$-trail $T_2$    that share a common vertex $u$, where $x_i\in X$ and $y_i\in Y$ for $i=1,2$, with $u$ possibly being  in $\{x_{1},x_{2},y_{1},y_{2}\}$.     Let  $T'_1$ be the $(x_1,y_2)$-trail consisting of  the edges    from $x_1$ to $u$ along $T_1$ and   from $u$ to $y_2$ along $T_2$; let $T'_2$ be the $(y_1,x_2)$-trail  consisting of  the edges    from $y_1$ to $u$ along $T_2$ and   from $u$ to $x_2$ along $T_1$.
Then $(\mathcal T\setminus\{T_1,T_2\})\cup \{T'_1,T'_2\}$ is also a good  open trail decomposition of $G$. However,    $T'_1, T'_2$ are both $XY$-trails when   $u\notin\{x_{1},x_{2},y_{1},y_{2}\}$, and one  of $T'_1,T'_2$ is an    $XY$-trail  when   $u\in\{x_{1},x_{2},y_{1},y_{2}\}$,   a contradiction.
Thus,  $E$ admits a good open trail decomposition containing at least one  {$XY$}-{trail}.

 This completes the proof of Lemma \ref{lem5}.
\end{proof}
The following lemma is a general tool that could be useful in other contexts.

\begin{lemma} \label{lem6}  Let $G=(X\cup Y, E)$ be a   bipartite   graph without isolated vertices  or Eulerian components,  where  either each vertex in $Y$ has an even degree or there   exists a component with at least one odd-degree vertex in each of  $X$ and $Y$.   Suppose $L$  consists of $l_1$   $1$-labels and $l_2$   $2$-labels,    with    $l_1+l_{2}=|E|$ and $|l_1-l_{2}|\leq 2$. Then there   is a  labeling  $f$ on $G$ with $L$     such that the   following conditions hold (where  $l_{\mu,v}$ is the number of edges incident to $v$ that are labeled with ${\mu}$-labels, for $\mu=1,2$).   \begin{kst}
   \item[(i)]   Each vertex in $X$ receives a   $\{1,2\}$-sum.
\item[(ii)]    \vspace{1mm} If each vertex in $Y$ has an even   degree, then $l_{1,y}=l_{{2},y}$ for each vertex $y$  but at most one (supposed to be $y'$ if exists)  in $Y$, and   $|l_{{1},y'}-l_{{2},y'}|=2$. If there exists at least one  odd-degree vertex  in $Y$, then $|l_{{1},y}-l_{{2},y}|\leq 1$ for each vertex $y$  in $Y$.
    \item[(iii)]    \vspace{1mm} Let $W$ and $Y_{2}$  be the set of pendant  vertices and even-degree vertices   in $Y$, respectively; let      $Y_3=Y\setminus(W\cup Y_2)$.  If  $|Y_{3}|\leq1$,    $|Y_{2}\cup Y_{3}|\geq \frac{|V_G|}{2}$  and each vertex in $Y_2\cup Y_3$ has degree at least four, then $X\cup W$ admits at least one $\mu$-vertex sum for each $\mu\in\{1,2\}$.
       \end{kst}   \end{lemma}

\begin{proof}  It is sufficient to assume $l_1\geq l_{2}$.

  We show (i) and (ii) first. Let $2r$ be the   number of odd-degree vertices in $G$. By Lemma \ref{lem22},  $E$ decomposes into $r$ open trails $T_1,T_2,\ldots,T_{r}$,  where we suppose     $T_i$ is a $Y$-trail for $i\in[1,r_1]$,   an $XY$-trail for $i\in[r_1+1,r_1+r_2]$, and an $X$-trail for $i\in[r_1+r_2+1,r]$. Let
\begin{flalign*}  \ \ \ \ \ \ \ \
    T_i=\left\{
   \begin{array}{ll}
y_{i,1}x_{i,2}y_{i,2}\ldots x_{i,t_i}y_{i,t_i},& {\rm for}\ i\in[1,r_1],\vspace{1mm}  \\
 y_{i,1}x_{i,2}y_{i,2}\ldots x_{i,t_i}y_{i,t_i}x_{i,t_i+1},&  {\rm for}\ i\in[r_1+1,r_1+r_2],\\
 x_{i,1}y_{i,1}x_{i,2}y_{i,2}\ldots x_{i,t_i}y_{i,t_i}x_{i,t_i+1},&  {\rm for}\ i\in[r_1+r_2+1,r].
   \end{array}
  \right.
  &&&\end{flalign*}
     See Figure \ref{fig-4} for example.

 \begin{figure}[htbp]
 \centering
 %\psfrag{a}{(a)}
 %\psfrag{b}{(b)}
 \includegraphics[height=3.6cm]{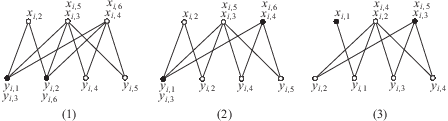}
\caption{\label{fig-4} (1) A $Y$-trail where $i\in[1,r_1]$;  (2) An $XY$-trail where $i\in[r_1+1,r_1+r_2]$; (3) An $X$-trail where $i\in[r_1+r_2+1,r]$.}
\end{figure}

     We choose the open trail decomposition of $E$  such that $r_2$ is maximum. If $r_1>0$, then there exists an odd-degree vertex in $Y$. According to our hypothesis,    there exists a component  $H$   of $G$ such that $H$  has    odd-degree  vertices    in both of   $X$ and $Y$. In this case,   Lemma \ref{lem5} implies  $r_2>0$.

        Label  $T_1,T_2,\ldots,T_{r}$ in the order one by one.    Label the edges of $T_i$ in the order along the trail starting from $y_{i,1}x_{i,2}$ (resp. $x_{i,1}y_{i,1}$) for   $i\in[1,r_1+r_2]$ (resp. for $i\in[r_1+r_2+1,r]$). We label  as many edges as possible    under the following rules.
                  \begin{kst}
        \item[(R1)]
        For $(i,j)\in ([1,r_1]\times [2,t_i-1])\cup ([r_1+1,r_1+r_2]\times [2,t_i])\cup([r_1+r_2+1,r]\times [1,t_i])$, label $y_{i,j}x_{i,j+1}$  with a $({3-\mu})$-label  whenever    $x_{i,j}y_{i,j}$ has been labeled with a ${\mu}$-label ($\mu=1,2$); see Figure \ref{fig-1}, or see Figures \ref{fig-6}, \ref{fig-7} and   \ref{fig-8}, for example.

 \begin{figure}[htbp]
 \centering
 %\psfrag{a}{(a)}
 %\psfrag{b}{(b)}
 \includegraphics[height=2.6cm]{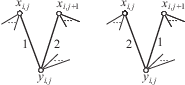}
\caption{\label{fig-1} The labeling rule of (R1).}
\end{figure}

      \item[(R2)] \vspace{1mm} Let $\tilde{l}_{\mu,i}$ be the number  of $\mu$-labels that have already been used right before labeling $y_{i,1}x_{i,2}$ for $\mu=1,2$ and $i\in[2,r_1+r_2]$.  \begin{kst}
\item  \vspace{1mm} Label $y_{1,1}x_{1,2}$ with a $1$-label.  For $i\in [2,r_1]$, label $y_{i,1}x_{i,2}$ with a $1$-label if $\tilde{l}_{1,i}\leq \tilde{l}_{2,i}$, and   with a $2$-label otherwise.
    \item  \vspace{1mm} Suppose   $l_1=l_2+\iota$ ($\iota=0,1,2$), then for $i\in [r_1+1,r_1+r_2]$, label $y_{i,1}x_{i,2}$ with a $1$-label if $\tilde{l}_{1,i}\leq \tilde{l}_{2,i}+\iota$, and   with a $2$-label otherwise.
        \item  \vspace{1mm}  For $i\in [r_1+r_2+1,r]$, label  $x_{i,1}y_{i,1}$    with a $1$-label.
    \end{kst}
 \item[(R3)]    \vspace{1mm}   Let $(i,j)\in [1,r]\times[2,t_i]$. Denote by $h(x_{i,j})$ the number of edges incident to $x_{i,j}$ that have been labeled  right before choosing a  label for $x_{i,j}y_{i,j}$. Then $h(x_{i,j})$ is odd  if  $d(x_{i,j})$ is even.  Suppose $y_{i,j-1}x_{i,j}$ has been labeled with a $\mu$-label ($\mu=1,2$).
 \begin{kst}
\item     \vspace{1mm} If  $x_{i,j}$ has an odd degree, or $h(x_{i,j})\leq d(x_{i,j})-3$,  then  label $x_{i,j}y_{i,j}$ with a $({3-\mu})$-label; see Figure \ref{fig-2}, or see $x_{i,3}$ in Figure \ref{fig-6}, or see $x_{i,4}$ in  Figure \ref{fig-7}, for example.

 \begin{figure}[htbp]
 \centering
 %\psfrag{a}{(a)}
 %\psfrag{b}{(b)}
 \includegraphics[height=2.6cm]{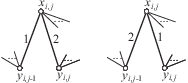}
\caption{\label{fig-2} The first labeling rule of (R3) when $x_{i,j}$ has an odd degree, or $h(x_{i,j})\leq d(x_{i,j})-3$.}
\end{figure}

    \item    \vspace{1mm} If  $x_{i,j}$ has  an even degree and $h(x_{i,j})= d(x_{i,j})-1$,   label $x_{i,j}y_{i,j}$ with a ${\mu}$-label; see Figure \ref{fig-3}, or   see $x_{i,5}$ in  Figure \ref{fig-6}, for example.
        \begin{figure}[htbp]
 \centering
 %\psfrag{a}{(a)}
 %\psfrag{b}{(b)}
 \includegraphics[height=2.6cm]{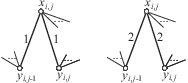}
\caption{\label{fig-3} The second labeling rule of (R3) when $x_{i,j}$ has an even degree  and  $h(x_{i,j})= d(x_{i,j})-1$.}
\end{figure}
\end{kst}
         \end{kst}

  Clearly, if $f$  completes without being suspended, then by (R3), one has that $|l_{1,x}-l_{{2},x}|=1$ or $2$ for each $x\in X$,  and then (i) holds; by   (R1), one   has that $|l_{1,y}-l_{{2},y}|\leq1$ for each $y\in Y$, and then  (ii) holds.

  We first analyze under what conditions $f$ fails to complete.

     For $\mu=1,2$ and $i\in[1,r]$, denote by $l_{\mu,T_i}$ the number of $\mu$-labels we need to assign in $T_i$ under the above labeling rules.   Then \begin{flalign}\label{eqmit} \ \ \ \ \ \ \ \
    l_{1,T_i}-l_{2,T_i}=\left\{
   \begin{array}{ll}
-2,0\ {\rm or}\ 2,&{\rm for}\  i\in[1,r_1],\vspace{1mm}  \\
  -1\ {\rm or}\ 1,&{\rm for}\  i\in[r_1+1,r_1+r_2],\vspace{1mm}\\
 0,& {\rm for}\ i\in[r_1+r_2+1,r].
   \end{array}
  \right.
  &&&\end{flalign}
 See Figures \ref{fig-6}, \ref{fig-7} and \ref{fig-8}, for example.
  \begin{figure}[htbp]
 \centering
 %\psfrag{a}{(a)}
 %\psfrag{b}{(b)}
 \includegraphics[height=3.6cm]{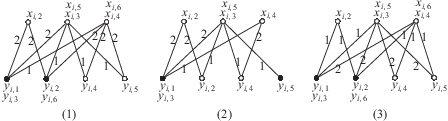}
\caption{\label{fig-6} (1) A labeled $Y$-trail where $l_{1,T_i}-l_{2,T_i}=-2$; (2) A  labeled $Y$-trail where $l_{1,T_i}-l_{2,T_i}=0$; (3) A labeled $Y$-trail where $l_{1,T_i}-l_{2,T_i}=2$.}
\end{figure}
\begin{figure}[htbp]
 \centering
 %\psfrag{a}{(a)}
 %\psfrag{b}{(b)}
 \includegraphics[height=3.6cm]{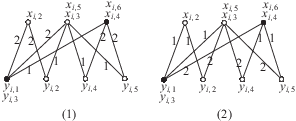}
\caption{\label{fig-7} (1) A labeled $XY$-trail where $l_{1,T_i}-l_{2,T_i}=-1$; (2) A  labeled $XY$-trail where $l_{1,T_i}-l_{2,T_i}=1$.}
\end{figure}
   \begin{figure}[htbp]
 \centering
 %\psfrag{a}{(a)}
 %\psfrag{b}{(b)}
 \includegraphics[height=3.2cm]{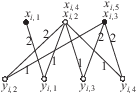}
\caption{\label{fig-8}   A labeled $X$-trail  where $l_{1,T_i}-l_{2,T_i}=0$.}
\end{figure}

   Denote by $l_\mu'$ (resp. $l_\mu''$  or $l_\mu'''$)   the number of   $\mu$-labels  that we need to
     finish  labeling $\{T_{i}|i\in[1,r_1]\}$ (resp. $\{T_{i}|i\in[r_1+1,r_1+r_2]\}$ or $\{T_{i}|i\in[r_1+r_2+1,r]\}$) under the labeling rules, for $\mu=1,2$. One has that $l_1'''=l_2'''$ by the third item of (\ref{eqmit}), and so the labeling process completes without being suspended if and only if $l_1-(l_1'+l_1'')=l_2-(l_2'+l_2'')$.

     Let
      \begin{flalign*}  \ \ \ \ \ \ \ \ \mathcal T'=\{T_i|i\in[1,r_1],l_{1,T_i}-l_{2,T_i}\neq0\}=\{T_{i_1},T_{i_2},\ldots,T_{i_{r'_1}}\},
    &&&\end{flalign*}  where $i_1< i_2< \ldots< i_{r'_1}$.  By the first item of (\ref{eqmit}) and the first item of (R2), one has that $l_{1,T_{i_j}}-l_{2,T_{i_j}}=2$
  for  each odd integer $j$ in $[1,r_1']$ and $l_{1,T_{i_j}}-l_{2,T_{i_j}}=-2$
  for each even integer $j$ in $[1,r_1']$. This  implies
   \begin{flalign}\label{leq1} \ \ \ \ \ \ \ \
    \left\{
   \begin{array}{ll}
l'_1=l'_2,&{\rm when}\  r_1'\ {\rm is\ even},\vspace{1mm}  \\
 l'_1=l'_2+2,&{\rm when}\  r_1'\ {\rm is\ odd}\vspace{1mm}.
   \end{array}
  \right.
  &&&\end{flalign}
   Note that $r_2$ is odd if and only if $l_1=l_2+1.$ By (\ref{eqmit}), (\ref{leq1}) and the second item of  (R2),  the following  hold.
   \begin{kst}
   \item \vspace{1mm} If $l_1=l_2$ (then $r_2$ is even) and $l'_1=l'_2$, then $l_{1,T_{r_1+j}}-l_{2,T_{r_1+j}}=1$
  for  each odd integer $j$ in $[1,r_2]$ and $l_{1,T_{r_1+j}}-l_{2,T_{r_1+j}}=-1$
  for each even integer $j$ in $[1,r_2]$. This implies $l''_1=l''_2$ and so  $l_1-(l_1'+l_1'')=l_2-(l_2'+l_2'')$.

   \item \vspace{1mm}
   If $l_1=l_2$   and $l'_1=l'_2+2$ (then $r_1>0$ and so $r_2>0$), then   $l_{1,T_{r_1+j}}-l_{2,T_{r_1+j}}=-1$
  for $j=1,2$,  $l_{1,T_{r_1+j}}-l_{2,T_{r_1+j}}=1$
  for each odd integer $j$ in $[3,r_2]$, and $l_{1,T_{r_1+j}}-l_{2,T_{r_1+j}}=-1$
  for each even integer $j$ in $[3,r_2]$. This implies $l''_1=l''_2-2$  and so $l_1-(l_1'+l_1'')=l_2-(l_2'+l_2'')$.

  \item \vspace{1mm} If $l_1=l_2+1$ (then $r_2$ is odd) and $l'_1=l'_2$, then
     $l_{1,T_{r_1+j}}-l_{2,T_{r_1+j}}=1$
  for $j=1$ and each even integer $j$ in $[2,r_2]$ and $l_{1,T_{i_j}}-l_{2,T_{i_j}}=-1$
  for each odd integer $j$ in $[3,r_2]$. This implies $l''_1=l''_2+1$   and so $l_1-(l_1'+l_1'')=l_2-(l_2'+l_2'')$.

   \item \vspace{1mm}
   If $l_1=l_2+1$ and $l'_1=l'_2+2$, then   $l_{1,T_{r_1+j}}-l_{2,T_{r_1+j}}=-1$
  for each odd integer $j$ in $[1,r_2]$, and   $l_{1,T_{r_1+j}}-l_{2,T_{r_1+j}}=1$
  for each even integer $j$ in $[1,r_2]$. This implies $l''_1=l''_2-1$ and so $l_1-(l_1'+l_1'')=l_2-(l_2'+l_2'')$.

  \item \vspace{1mm}
  If  $l_1=l_2+2$ (then $r_2$ is even) and  $l'_1=l'_2$, then $l_{1,T_{r_1+j}}-l_{2,T_{r_1+j}}=1$
  for $j=1,2$ and   each odd integer $j$ in $[3,r_2]$, and $l_{1,T_{r_1+j}}-l_{2,T_{r_1+j}}=-1$
  for each even integer $j$ in $[3,r_2]$. This implies $l''_1=l''_2+2$ when $r_2>0$. That is, $l_1-(l_1'+l_1'')=l_2-(l_2'+l_2'')$ when $r_2>0$.

  \item \vspace{1mm} If  $l_1=l_2+2$ and  $l'_1=l'_2+2$ (then $r_1>0$ and so $r_2>0$), then $l_{1,T_{r_1+j}}-l_{2,T_{r_1+j}}=1$
  for each odd integer $j$ in $[1,r_2]$, and $l_{1,T_{r_1+j}}-l_{2,T_{r_1+j}}=-1$
  for each even integer $j$ in $[1,r_2]$. This  implies $l''_1=l''_2$ and so $l_1-(l_1'+l_1'')=l_2-(l_2'+l_2'')$.
  \end{kst}

Therefore, $f$ fails to complete if and only  if   $r_2=0$  (then $r_1=0$  and so $l'_1=l'_2=0$ and $r_3>0$ since $r_1+r_2+r_3>0$) and $l_1=l_2+2$.
  In that case,  we   just  label each of $x_{r,1}y_{r,1}$ and  $y_{r,1}x_{r,2}$ with a    1-label, then   label $E_{T_r}\setminus \{x_{r,1}y_{r,1},y_{r,1}x_{r,2}\}$ under   (R1) and (R3). Then    the labeling process   completes  without being suspended anymore.  One has (i) holds since we still obey (R3). What is more, one has    $l_{1,y_{r,1}}-l_{{2},y_{r,1}}=2$, and    $l_{1,y}=l_{{2},y}$ for each     $y\in Y_2\setminus\{y'\}$, where we set $y'=y_{r,1}$. This implies  (ii) also  holds.
   Thus, (i) and (ii) hold.

 \vspace{2mm} For (iii), let $f$ be a  labeling of $G$ described above in the proofs of   (i) and (ii). Note that $|l_{1,x}-l_{2,x}|\in\{1,2\}$    for each $x\in X$ under $f$ by (R3).  Let $q_\mu$   be the number of  $\mu$-vertex sums in $X\cup W$ under $f$ for $\mu=1,2$. Denote by $q=|X\cup W|$. Clearly, $q_1+q_2=q$ and $q\geq |X|\geq4$ since $\delta_G(Y_2\cup Y_3)\geq4$. If $\min\{q_1,q_2\}>0$, then $f$ satisfies  (i), (ii), (iii) and we are done.
   Suppose (without loss of generality)   that $q_2=0$.   That is, $q_1=q$ and  $|q_1-q_2|=q$.   We would do some switching on $f$ to satisfy (iii) as follows.

      Let $H=G[X\cup Y_2\cup Y_3]$.  Then
         \begin{flalign*}  \ \ \ \ \ \ \ \
         \sum_{x\in X}d_H(x)=\sum_{y\in Y_2\cup Y_3}d_H(y)\geq 4|Y_2\cup Y_3|\geq 4|X|,&&&\end{flalign*}
      since $\delta_H(Y_2\cup Y_3)=\delta_G(Y_2\cup Y_3)\geq4$  and $|Y_2\cup Y_3|\geq \frac{|V_G|}{2}$ (so $|Y_2\cup Y_3|\geq |X|$). It follows that   $\Delta_H(X)\geq 4$. Let $u$ be a vertex in $X$ with $d_{H}(u)\geq4$. Recall that $|l_{1,u}-l_{2,u}|\in\{1,2\}$ under $f$.    If  $d_{G}(u)\geq5$,  then $l_{2,u}\geq2$;  if $d_{G}(u)=d_{H}(u)=4$, then  $l_{1,u}=1$ and $l_{2,u}=3$ since $u$ receives  a 1-vertex sum by assumption.  So we always have  $l_{2,u}\geq2$. Let $uy_1,uy_2$ be two edges  in $H$, both labeled with a 2-label. Note that each edge in $E_G(u)\setminus E_H(u)$ is assigned with a 1-label since each vertex in $W$ receives a 1-vertex sum by our assumption. So   $y_1,y_2\in Y_2\cup Y_3$.  Recall that $|l_{1,y}-l_{2,y}|\leq2$ for each $y\in Y_2\cup Y_3$ by (ii), and     $\delta_H(Y_2\cup Y_3)=\delta_G(Y_2\cup Y_3)\geq4$.   So  there exists  at least one    edge  $u_{i}y_i$ in $H$  labeled with a $1$-label  for each $i\in\{1,2\}$. Let $f'$ be the labeling obtained from $f$ by exchanging the labels on $uy_i$ and   $u_{i}y_i$ for each $i\in\{1,2\}$, and let   $q'_\mu$   be the number of $\mu$-vertex sums in $X\cup W$ under $f'$ for $\mu=1,2$.   Then each of     $u,u_{1},u_{2}$ receives a  2-vertex sum under $f'$ and all other vertex sums remain unchanged after switching. So (i) and (ii) still hold under $f'$. On the other hand, one has $q'_2=3$ and $q'_1=q-3\geq1$ since   $q\geq4$. This implies
   $\min\{q'_1,q'_2\}>0$ and   $f'$ satisfies  (i), (ii) and (iii).

  This completes the proof of Lemma \ref{lem6}.
 \end{proof}

\begin{lemma} \label{lem4}  Let $G$ be   a  connected  graph that admits a  vertex $v$ incident to at least    three non-cut  edges.  Then   there exist  two edges $e_1,e_2$  incident to $v$   such that $G-\{e_1,e_2\}$ is connected.
\end{lemma}

\begin{proof}  Let $E'(v)$  be the set of  non-cut edges incident to $v$. Let $e$ be an arbitrary edge in $E'(v)$, and suppose  $e$  is contained in some block.
 Clearly, $v$ is   contained in the block. If $e$ is the unique edge
 in $E'(v)$ that is contained in the block, then $e$ will be a cut edge, a contradiction. So  each block that contains some edge in $E'(v)$ contains   another  edge in $E'(v)$. If all edges in $E'(v)$ are included in some block, then arbitrary two  edges in $E'(v)$ are the desired   $e_1, e_2$.  If there are  two blocks including  edges in $E'(v)$, then      arbitrary two edges in $E'(v)$ that are in  distinct blocks are the desired   $e_1,e_2$.
\end{proof}

\begin{lemma} \label{lem3}    Let $J=\{p+3i-2,p+3i-1|i\in[1, k]\}$ and $J'=\{p+3i-1,p+3i+1|i\in[1, k]\}$ where $p$ is a 0-label.

\begin{kst} \item[(i)]   If $k$ is odd, then the labels in  $J$ can be paired such that the pair sums are exactly the $k$ continuous 0-sums in   $\{2p+\frac{3(k-1)}{2}+3i|i\in[1, k]\}$; the labels in  $J'$ can be paired such that the pair sums are exactly the $k$ continuous 0-sums in   $\{2p+\frac{3(k+1)}{2}+3i|i\in[1, k]\}$.

\item[(ii)] \vspace{2mm} If $k$ is positive even,  then the labels in  $J$ can be paired such that the pair sums are exactly the $k-1$ continuous 0-sums  in  $\{2p+3(\frac{k}{2}+1)+3i|i\in[1, k-1]\}$ and the other one 0-sum $2p+3$; the labels in  $J'$ can be paired such that the pair sums are exactly the $k-1$ continuous 0-sums  in  $\{2p+3(\frac{k}{2}+2)+3i|i\in[1, k-1]\}$ and the other one 0-sum $2p+6$.
    \end{kst}
 \end{lemma}
\begin{proof} (i) Let $k=2q+1$. In $J$, pair $\{p+3i-2,p+3(q+i)-1\}$ for $i\in [1,q+1]$ and $\{p+3i-1,p+3(q+1+i)-2\}$ for $i\in [1,q]$.   In $J'$, pair $\{p+3i-1,p+3(q+i)+1\}$ for $i\in [1,q+1]$ and $\{p+3i+1,p+3(q+1+i)-1\}$ for $i\in [1,q]$. These pairings achieve the desired results; see Figure \ref{fig-9} (1) for example.

 (ii) Let $k=2q$. In $J$, pair $\{p+1,p+2\}$, pair  $\{p+3(i+1)-2,p+3(q+i)-1\}$ for $i\in [1,q]$ and pair $\{p+3(i+1)-1, p+3(q+1+i)-2\}$ for $i\in [1,q-1]$. In $J'$,  pair $\{p+2,p+4\}$, pair  $\{p+3(i+1)-1,p+3(q+i)+1\}$ for $i\in [1,q]$ and pair $\{p+3(i+1)+1,p+3(q+1+i)-1\}$ for $i\in [1,q-1]$. These pairings achieve the desired results; see Figure \ref{fig-9} (2) for example.
\end{proof}

\begin{figure}[htbp]
 \centering
 %\psfrag{a}{(a)}
 %\psfrag{b}{(b)}
 \includegraphics[height=3.8cm]{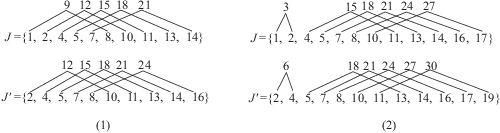}
\caption{\label{fig-9} (1) The  pairings of $J$ and $J'$ when $p=0$ and $k=5$, such that the pair sums are continuous 0-sums; (2) The  pairings of $J$ and $J'$ when $p=0$ and $k=6$, such that the pair sums are almost continuous 0-sums.}
\end{figure}

\section{Proof of the main result}
 Let $G=(V,E)$ be a  bipartite graph  with {minimum degree at least 15}. The proof of the antimagicness of $G$ divides  into two   parts. The first   part  is to construct  a partition of $V$,  and a decomposition of $E$,   where we       find  a special    subset of  $E$ of a certain size that   induces a subgraph without Eulerian    components. The second part     is  to carefully   define  an antimagic     labeling for   $G$, where some    vertices receive pairwise distinct 0-vertex sums, while others receive  pairwise distinct $\{1,2\}$-vertex sums, based on the partitions of $V$ and  $[|E|]$, and the decomposition of $E$.

\begin{proof}[\textbf{Proof of Theorem \ref{thm1}}]   Consider a bipartite graph $G=(A\cup B,E)$    of size $m$ with {minimum degree at least 15}. First, we give a vertex partition and edge decomposition of $G$ as follows.

\vspace{2mm} \noindent\textbf{1. The partition of $A\cup B$, and the decomposition of $E$.}

\vspace{2mm}\noindent \textbf{1.1. The definitions  of $X, Y, Z,  I_1$ in $G$, and that of  $Y_{{\rm odd}}, Y_{{\rm even}}, n_X, n_Y, n^{{\rm odd}}_{Y}$ and $n^{{\rm even}}_{Y}$ in $G[X,Y]$.}

\vspace{2mm}Suppose $X\cup Y$  is the partition of $A\cup B$ described in  Lemma \ref{lem1}, where $Y$ is an independent set and  there exists a matching   $M$  in $G[X,Y]$ that saturates $X$. Define $Z\subseteq Y$ as  the set of   vertices that are not saturated by $M$. Denote by $I_1$    the isolated vertex set   in $G[X]$. Let $Y_{{\rm odd}}$ (resp. $Y_{{\rm even}}$)   represent the set of vertices in $Y$ with  odd degree (resp. even degree). See Figure \ref{fignewnew1}.
\begin{figure}[htbp]
 \centering
 %\psfrag{a}{(a)}
 %\psfrag{b}{(b)}
 \includegraphics[height=3.7cm]{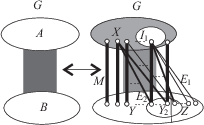}
\caption{\label{fignewnew1}    The vertex partitions of $A\cup B$ to $X\cup Y$,   where the matching $M$ (the thickest edges) saturates $X$ and $Y\setminus Z$, the edge set $E_1$ (the thinnest edges) exactly covers $Z$,   $E_2$ (the medium-thick edges) exactly  covers $Y_{{\rm even}}$, and $I_1$ is the isolated vertex set   in $G[X]$. White regions   represent independent sets.}
\end{figure}

 Denote by $n_X\triangleq|X|$, $n_Y\triangleq|Y|$,    $n^{{\rm odd}}_{Y}\triangleq|Y_{{\rm odd}}|$, and $n^{{\rm even}}_{Y}\triangleq|Y_{{\rm even}}|$. Clearly, $n_X\leq n_Y$.

 \vspace{2mm}\noindent \textbf{1.2. The definitions and findings of $E_1, E_2$ and $G_0$, in $G[X,Y]-M$.}

  \vspace{2mm} Let $E_1$ and $E_2$ be two   disjoint edge sets    in $G[X,Y]-M$, such that $E_1$  (resp. $E_2$) has a size $|Z|$ (resp. $|Y_{{\rm even}}|$) and covers each vertex in $Z$ (resp. $Y_{{\rm even}}$) exactly once; see Figure \ref{fignewnew1}.

Specially,  we choose $E_1$ and $E_2$ carefully, such that each vertex in $I_1$ has a positive degree in $G_0\triangleq G[X,Y]-(M\cup E_1\cup E_2)$, as follows.

 \begin{kst}
 \item[1)]    If  the set $I_0$ of  vertices in $I_1$ that are isolated in  $G_0$ is empty, then we are done.

 \item[2)]\vspace{2mm} If  $I_0\neq\emptyset$,  then we do some switchings  on $E_1\cup E_2$, such that $|I_0|$ decreases,   as follows.

  Let $x_1'$ be an arbitrary vertex in $I_0$.  By the definition of $I_0$, one has  $|(E_1\cup E_2)(x_1')|=d_G(x_1')-1\geq 14$. So there exists at least one edge $x_1'y'_1$ in  $E_1\cup E_2$.  Observe that   $d_{G_0}(y)\geq14$ which  is even for each $y\in Y$. So there exists  a   neighbor $x'_2$ of $y'_1$ in $G_0$ (then $x'_2\notin I_0$ by the definition of $I_0$).

  A  trail $T=x_1y_1x_2y_2\ldots x_ty_tx_{t+1}$ of   length $2t$   in $G[X,Y]-M$ is called an     even $(I_0,E_1\cup E_2, G_0)$-\emph{alternating} trail, if $x_1\in I_0$,  the edge set $\{x_iy_{i}|i\in[1,t]\}\subseteq E_1\cup E_2$, and  $\{y_ix_{i+1}|i\in[1,t]\}\subseteq E_{G_0}$; see Figure \ref{fignew1}, for example. By the discussions above,  $x_1'y'_1x_2'$ is an even $(I_0,E_1\cup E_2, G_0)$-{alternating} trail. So this definition is well-defined.

  \begin{figure}[htbp]
 \centering
 %\psfrag{a}{(a)}
 %\psfrag{b}{(b)}
 \includegraphics[height=3.5cm]{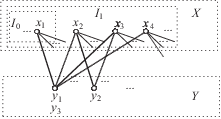}
\caption{\label{fignew1} The even $(I_0,E_1\cup E_2, G_0)$-{alternating} trail $T=x_{1}y_{1}x_{2}y_{2}x_{3}y_{3}x_{4}$, where $\{x_{1}y_{1},x_{2}y_{2},x_{3}y_{3}\}\subseteq E_1\cup E_2$,  and $\{y_{1}x_{2},y_{2}x_{3},y_{3}x_{4}\}\subseteq E_{G_0}$. Solid lines represent the edges of $E_1\cup E_2$, while dashed lines represent the edges of $E_{G_0}$.}
\end{figure}

  Let $\mathcal T=\{T_{1},T_{2},\cdots,T_{\iota}\}$ be the set of all even $(I_0,E_1\cup E_2, G_0)$-\emph{alternating} trails, where $T_{j}=x_{j,1}y_{j,1}x_{j,2}y_{j,2}\ldots x_{j,t_{j}}y_{j,t_{j}}x_{j,t_{j}+1}$ for $j\in[1,\iota]$. Let    $G_{\mathcal T}$ be the graph induced by $\{E_T|T\in \mathcal T\}$. Denote by $X_{\mathcal T}=V_{G_{\mathcal T}}\cap X$ and $Y_{\mathcal T}=V_{G_{\mathcal T}}\cap Y$.  By definition of $\mathcal T$,  $I_0=\{x_{j,1}|j\in [1,\iota]\}$, $X_{\mathcal T}\setminus I_0=\{x_{j,i}|j\in[1,\iota],i\in[2,t_{j}+1]\},$ and $Y_{\mathcal T}=\{y_{j,i}|j\in[1,\iota],i\in[1,t_{j}]\}.$

 Clearly, $(E_1\cup E_2)(I_0)\subseteq E_{G_{\mathcal T}}$ by definition. What is more, we show  that $(E_1\cup E_2)(X_{\mathcal T}\setminus I_0)\subseteq E_{G_{\mathcal T}}$ also holds, as follows. Suppose to the contrary that  there exist  $j\in[1,\iota]$,  $i\in[2,t_{j}+1]$, and an edge   $x_{j,i}y$ in  $(E_1\cup E_2)(x_{j,i})$, such that  $x_{j,i}y\notin E_{G_{\mathcal T}}$. Recall that $|(E_1\cup E_2)(y)|\leq2$ and    $d_{G_0}(y)\geq 14$. So the trail  $T_j$ goes through  $y$ at most once, and there exist  at least 13 edges incident to $y$ that are not in  $T_j$. Let $yx$ be arbitrary one of those 13 edges. Then by adding the two edges $x_{j,i}y, yx$ to  $T_j$, we obtain a new trail that is also an even $(I_0,E_1\cup E_2, G_0)$-{alternating} trail but is not contained by $\mathcal T$. This is a contradiction to the definition of $\mathcal T$. Therefore, for each $j\in[1,\iota]$ and   each $i\in[2,t_{j}+1]$,  one has that each   edge   in  $(E_1\cup E_2)(x_{j,i})$ is in   $E_{G_{\mathcal T}}$. That is, $(E_1\cup E_2)(X_{\mathcal T}\setminus I_0)\subseteq E_{G_{\mathcal T}}$.  Thus, \begin{flalign*}  \ \ \ \ \ \ \ \ \ \ \ \ \ \ \ \ (E_1\cup E_2)(X_{\mathcal T})\subseteq E_{G_{\mathcal T}}. &&&\end{flalign*}
  Then each edge in $(E_1\cup E_2)(X_{\mathcal T})$ is incident to some vertex in $Y_{\mathcal T}$. This   implies that $(E_1\cup E_2)(X_{\mathcal T})\subseteq (E_1\cup E_2)(Y_{\mathcal T})$, and so
  \begin{flalign} \label{xinxin1} \ \ \ \ \ \ \ \ \ \ \ \ \ \ \ \ |(E_1\cup E_2)(X_{\mathcal T})|\leq |(E_1\cup E_2)(T_{\mathcal T})|.&&&\end{flalign}

 If for each    $j\in[1,\iota]$ and each $i\in[2,t_{j}]$, one has that $x_{j,i}\in (I_1\setminus I_0)$ and $d_{G_0}(x_{j,i})=1$ (that is, $x_{j,i}$  has  a unique neighbor $y_{j,i-1}$ in $Y_{\mathcal T}$ in $G_0$, and so $|E_1\cup E_2(x_{j,i})|\geq 13$),  then \begin{flalign}\label{xinxin2} \ \ \ \ \ \ \ \ \ \ \ \ \ \ \ \ |Y_{\mathcal T}|\leq |X_{\mathcal T}\setminus I_0|.&&&\end{flalign}
Recall that  for each $y_{j,i}$ in $Y_{\mathcal T}$,  $|(E_1\cup E_2)(x_{j,i})|\leq 2$.  It follows from (\ref{xinxin2}) that \begin{flalign*}  \ \ \ \ \ \ \ \ \ \ \ \ \ \ \ \ |(E_1\cup E_2)(Y_{\mathcal T})|\leq 2|Y_{\mathcal T}|\leq 2|X_{\mathcal T}\setminus I_0|<13|X_{\mathcal T}\setminus I_0|<13|X_{\mathcal T}|\leq|(E_1\cup E_2)(X_{\mathcal P})|.&&&\end{flalign*}
This   contradicts to (\ref{xinxin1}). So there exists some    $j\in[1,\iota]$ and some $i\in[2,t_{j}+1]$, such that  $d_{G_0}(x_{j,i})\geq2$ or $x_{j,i}\notin (I_1\setminus I_0)$ (that is,  $x_{j,i}\notin I_1$ since $x_{j,i}\notin I_0$).

 Now, switch  $E_1\cup E_2$ as follows: delete $\{x_{j,i}y_{j,i}|i\in[1,i-1]\}$ from $E_1\cup E_2$ and add
  $\{y_{j,i}x_{j,i+1}|i\in[1,i-1]\}$ to $E_1\cup E_2$.
 See Figures \ref{fignew2} and \ref{fignew3}, for example. Now $x_{j,i}$ is not isolated in     $G[X,Y]-(M\cup E_1\cup E_2)$ anymore, and so the   number of isolated vertices in  $G[X,Y]-(M\cup E_1\cup E_2)$ decreases.  By a similar discussion, we can inductively  switch $E_1\cup E_2$ until $I_0$ becomes empty.   \end{kst}
 \begin{figure}[htbp]
 \centering
 %\psfrag{a}{(a)}
 %\psfrag{b}{(b)}
 \includegraphics[height=3.5cm]{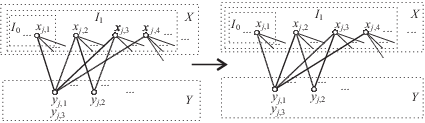}
\caption{\label{fignew2} An example for the switching of   $E_1\cup E_2$ when   $d_{G_0}(x_{j,4})\geq2$, in the even $(I_0,E_1\cup E_2, G_0)$-{alternating} trail $T_j=x_{j,1}y_{j,1}x_{j,2}y_{j,2}x_{j,3}y_{j,3}x_{j,4}$. Solid lines represent the edges of $E_1\cup E_2$.}
\end{figure}
\begin{figure}[htbp]
 \centering
 %\psfrag{a}{(a)}
 %\psfrag{b}{(b)}
 \includegraphics[height=3.5cm]{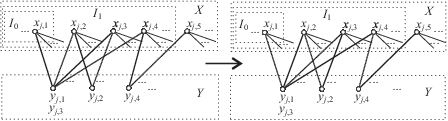}
\caption{\label{fignew3} The switching of  $E_1\cup E_2$  when $x_{j,5}\notin I_1$, in the even $(I_0,E_1\cup E_2, G_0)$-{alternating} trail $T_j=x_{j,1}y_{j,1}x_{j,2}y_{j,2}x_{j,3}y_{j,3}x_{j,4}y_{j,4}x_{j,5}$. Solid lines represent the edges of $E_1\cup E_2$.}
\end{figure}

 Thus, we have  successfully found the edge sets $E_1$ and $E_2$, ensuring that each vertex in $I_1$ has a positive degree in $G_0$.

\vspace{2mm} \noindent
\textbf{1.3. The definitions of $I_2, I_{2,1}, G_1, m_1$ in $G_0$,   that of $F_i,  k_i\ (i\in[1,3]), m_2$ in $G[X]$, and that of $l_\mu, l_{1,\mu}, l_{2,\mu}\ (\mu=0,1,2), O$ and $J$.}

\vspace{2mm}Let    $I_2$   be the isolated   vertex set  in $G_0$,   $I_{2,1}\subseteq I_2$ be the isolated vertex set  in $G[I_2]$, and
\begin{flalign*}  \ \ \ \ \ \ \ \ G_1\triangleq G_0[X\setminus I_2, Y];&&&\end{flalign*} see   Figure \ref{fignewnew2}.
 \begin{figure}[htbp]
 \centering
 %\psfrag{a}{(a)}
 %\psfrag{b}{(b)}
 \includegraphics[height=3.7cm]{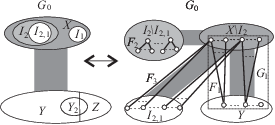}
\caption{\label{fignewnew2}     The vertex partitions  and edge decompositions of $G_0$, where each of  $F_1, F_2, F_3$ consists of star components. White regions   represent independent sets.}
\end{figure}

One has that $I_2\subseteq (X\setminus I_1)$ as $I_0=\emptyset$;   each vertex in $I_{2,1}$ is adjacent  to some vertex in  $X\setminus I_2$ by the definition of $I_{2,1}$. For each vertex $v$ in $X\setminus I_2$, $v$ has a neighbor $v'$ in $Y$ in $G_1$ by the definition of $I_2$. Let
 \begin{flalign*}  \ \ \ \ \ \ \ \ F_1=G[\{vv'|v\in X\setminus I_2\}].&&&\end{flalign*}
   Then $F_1$ is  a  forest of size $|X\setminus I_2|$  in $G_1$   that covers  each vertex in $X\setminus I_2$ exactly once, and  $F_1$ consists of  star components. See   Figure \ref{fignewnew2}.

    By definition of  $I_{2,1}$,  $G[I_2\setminus I_{2,1}]$ has no isolated vertex. Let
    $F'$ be  a forest in $G[I_2\setminus I_{2,1}]$ that covers each vertex in  $I_2\setminus I_{2,1}$. If  every edge in $F'$ is incident to some leaf, then let $F_2=F'$. If there exist edges that are not incident to any leaf, then step by step,
    we check and delete such edges, until we obtain a forest $F_2$  covering each vertex in  $I_2\setminus I_{2,1}$ and with  every edge in $F_2$ is incident to some leaf. Then      $F_2$ also  consists of     star components. See   Figure \ref{fignewnew2}.

Recall that each vertex in $I_{2,1}$ is adjacent  to some vertex in  $X\setminus I_2$ by the definition of $I_{2,1}$. For each vertex $v$ in  $I_{2,1}$, let $v'$ be a neighbor of $v$ in  $X\setminus I_2$. Let
 \begin{flalign*}  \ \ \ \ \ \ \ \ F_3=G[\{vv'|v\in I_{2,1}\}].&&&\end{flalign*} Then   $F_3$ is a   forest of size $|I_{2,1}|$ that covers each vertex in  $I_{2,1}$ exactly once and $F_3$ also  consists of star components.   See   Figure \ref{fignewnew2}.

  Let $k_i\triangleq|E_{F_i}|$ for $i=1,2,3$.
 Then one has \begin{flalign}\label{xeq4}  \ \ \ \ \ \ \ \
    \left\{
   \begin{array}{ll}k_1+k_2+k_3\leq n_X-1,\vspace{1mm}\\
   {k_2+k_3\leq |I_2|-1\leq n_X-15,}
   \end{array}
  \right.
 &&&\end{flalign}
   since
 $E_{F_1}\cup E_{F_2}\cup E_{F_3}$ induces a forest in $G[X]$, $E_{F_2}\cup E_{F_3}$ induces a forest in $G[I_2]$ and $|X\setminus I_2|\geq14$ {(note that $\delta_{G_1}(Y)=\delta_{G_0}(Y)\geq14$)}.
   For convenience, let $m_1\triangleq |E_{G_1}|$ and $m_2\triangleq|E_{G[X]}|$. Clearly, one has \begin{flalign*}  \ \ \ \ \ \ \ \ m=n_Y+n^{{\rm even}}_{Y}+m_1+m_2,&&&\end{flalign*}  since $E_G=(M \cup E_{1}) \cup E_{2}\cup E_{G_1}\cup E_{G[X]}$; and one has
   \begin{flalign}  \ \ \ \ \ \ \ \ \label{xeq2}m_1\geq14 n_Y,&&&\end{flalign}
   since   $\delta_{G_1}(Y)\geq14$.
  For $\mu=0,1,2$,  let $l_\mu$ be the number of $\mu$-labels in
 $[m]$,  $l_{1,\mu}$ be the number of $\mu$-labels in
 $[n_Y+n^{{\rm even}}_{Y}+m_1]$, and $l_{2,\mu}=l_\mu-l_{1,\mu}$. Let
$O$ be the set of 0-labels and $J$ be the set of $\{1,2\}$-labels in the set
$[m]$.  One has
  \begin{flalign}\label{xeq3}  \ \ \ \ \ \ \ \
    \left\{
   \begin{array}{ll}
|J|=\lceil\frac{2m}{3}\rceil =l_1+l_2=(l_{1,1}+l_{2,1})+(l_{1,2}+l_{2,2}),\vspace{1mm}  \\
  \lceil\frac{2(n_Y+n^{{\rm even}}_{Y}+m_1)}{3}\rceil =l_{1,1}+l_{1,2}. \end{array}
  \right.
  &&&\end{flalign}

\vspace{2mm}\noindent
\textbf{1.4.  The definition and finding of $G_3$ in $G[X]-(E_{F_2}\cup E_{F_3})$, and the definitions  of  $\varepsilon_1, E_3, m_{2,1}$ and $\gamma$ in $G[X]$.}

\vspace{2mm} We find an even subgraph $G_3$ in $G[X]-(E_{F_2}\cup E_{F_3})$ as follows.
\begin{kst}
 \item  If $k_2+k_3> l_{2,1}+l_{2,2}$, then let $G_3$ be an empty graph on  $X$.

   \item \vspace{2mm}   If $k_2+k_3\leq l_{2,1}+l_{2,2}$, then we construct   $G_3$ as follows.
     \begin{kst}
     \item[1)] If there is no cycle in  $G[X]-(E_{F_2}\cup E_{F_3})$, or each cycle in  $G[X]-(E_{F_2}\cup E_{F_3})$  has length greater than $(l_{2,1}+l_{2,2})-(k_2+k_3)$, then let $G_3$ be  an empty graph on $X$; otherwise, let $C_1$ be a cycle of length at most   $(l_{2,1}+l_{2,2})-(k_2+k_3)$.
        \item[2)]\vspace{2mm} Step by step, for $i\geq2$,  if there is no cycle in  $G[X]-[E_{F_2}\cup E_{F_3}\cup(\cup_{j\leq i-1}E_{C_j})]$, or each cycle in  $G[X]-[E_{F_2}\cup E_{F_3}\cup(\cup_{j\leq i-1}E_{C_j})]$  has length greater than $(l_{2,1}+l_{2,2})-(k_2+k_3+\sum_{j\leq i-1}|E_{C_{i}}|)$, then let $G_3$ be the graph induced by $\cup_{j\leq i-1}E_{C_j}$; otherwise, let $C_i$ be a cycle of length at most   $(l_{2,1}+l_{2,2})-(k_2+k_3+\sum_{j\leq i-1}|C_{i}|)$.
     \end{kst}That is, we find   an even graph $G_3$ with the maximum size  in $G[X]-(E_{F_2}\cup E_{F_3})$  such that   $k_2+k_3+|E_{G_3}|\leq l_{2,1}+l_{2,2}$.  \end{kst}
See Figure \ref{fignewnew3} (1).

Let   $E_3=E_{G_3}\cup E_{F_2}\cup E_{F_3}$. See Figure \ref{fignewnew3} (2).

\begin{figure}[htbp]
 \centering
 %\psfrag{a}{(a)}
 %\psfrag{b}{(b)}
 \includegraphics[height=4.1cm]{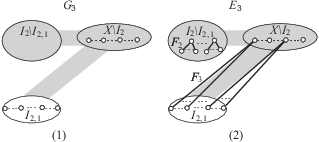}
\caption{\label{fignewnew3}  (1)  The even graph $G_3$ in  $G[X]-(E_{F_2}\cup E_{F_3})$ with a certain size   ($G_3$ may be empty). (2)  The edge set $E_3=E_{G_3}\cup E_{F_2}\cup E_{F_3}$ with certain size. }
\end{figure}

      Let $m_{2,1}\triangleq|E_3|$ and $\varepsilon_1\triangleq|E_{G_3}|$ (so $\varepsilon_1=0$ in the case when $k_2+k_3> l_{2,1}+l_{2,2}$).  Then $m_{2,1}=\varepsilon_1+k_2+k_3$. Denote by   \begin{flalign}  \ \ \ \ \ \ \ \ \label{xeq7} \gamma\triangleq (l_{2,1}
          +l_{2,2})-m_{2,1}. &&&\end{flalign}

      \begin{kst}

       \item  In the case when $k_2+k_3> l_{2,1}+l_{2,2}$, one has that  \begin{flalign} \ \ \ \ \ \ \ \ \ \ \ \ \ \ \ \ \label{new3} m_{2,1}=0+k_2+k_3\leq n_X-15.&&&\end{flalign} Furthermore, in this case,  \begin{flalign*} \ \ \ \ \ \ \ \ \ \ \ \ \ \ \ \ m_{2,1}\leq m_{2}=l_{2,0}+l_{2,1}+l_{2,2},&&&\end{flalign*}
   implying that
   \begin{flalign*} \ \ \ \ \ \ \ \ \ \ \ \ \ \ \ \ \gamma=(l_{2,1}
          +l_{2,2})-m_{2,1}\geq -l_{2,0}\geq -\frac{l_{2,1}+l_{2,2}}{2}> -\frac{k_2+k_3}{2}\geq-\frac{n_X-15}{2} >-\frac{n_X}{2}, &&&\end{flalign*}
   by  (\ref{new3}). So
    \begin{flalign} \ \ \ \ \ \ \ \ \ \ \ \ \ \ \ \ \label{new4} -\frac{n_X}{2}<\gamma<0,\ {\rm when}\ k_2+k_3> l_{2,1}+l_{2,2}.&&&\end{flalign}

     \item  \vspace{2mm} In the case when $k_2+k_3\leq l_{2,1}+l_{2,2}$, if $m_{2,1}\leq l_{2,1}+l_{2,2}-n_X$, then  the number of edges in $G[X]-E_3$ are more than $n_X$. This implies there  exists a cycle $C$ in $G[X]-E_3$ of length    at most $n_X$ since $|X|=n_X$. Then  $E_{G_{3}}\cup E_C$ induces an even graph and $m_{2,1}<m_{2,1}+|E_C|\leq l_{2,1}+l_{2,2}$, a contradiction to the choice  of $E_3$. Thus, \begin{flalign*} \ \ \ \ \ \ \ \ \ \ \ \ \ \ \ \  m_{2,1}> l_{2,1}+l_{2,2}-n_X,\ {\rm when}\ k_2+k_3\leq l_{2,1}+l_{2,2}.&&&\end{flalign*}
         So
         \begin{flalign} \ \ \ \ \ \ \ \ \ \ \ \ \ \ \ \ \label{new2} 0\leq\gamma=(l_{2,1}
          +l_{2,2})-m_{2,1}< n_X,\ {\rm when}\ k_2+k_3\leq l_{2,1}+l_{2,2}.&&&\end{flalign}
\end{kst}
From   (\ref{new4}) and (\ref{new2}), we conclude that
          \begin{flalign}  \ \ \ \ \ \ \ \ \label{xeq8}
             -\frac{n_X}{2}<\gamma< n_X.
 &&&\end{flalign}

\vspace{2mm}\noindent
\textbf{1.5.  The definition  and finding  of $E_4$ in $G_1$,    and the  definitions of $m_{1,1}$, $G_4$, and $Y^{odd}_{G_4}$ in $G_0[X\setminus I_2, Y\cup I_{2,1}]$.}

 \vspace{2mm}  Denote by
   \begin{flalign*}   \ \ \ \ \ \ \ \ m_{1,1}\triangleq \lceil\frac{2m}{3}\rceil-m_{2,1}.&&&\end{flalign*} By (\ref{xeq3}) and (\ref{xeq7}), one has that
          \begin{flalign}\label{xeq6}  \ \ \ \ \ \ \ \
        m_{1,1}=(l_{1,1}
          +l_{1,2})+(l_{2,1}
          +l_{2,2}-m_{2,1})
          =\lceil\frac{2(m_1+n_Y+n^{{\rm even}}_{Y})}{3}\rceil+\gamma. &&&\end{flalign}
           Then on one hand, by (\ref{xeq8}), (\ref{xeq6}), and the fact that $\max\{n_X,n^{{\rm even}}_{Y}\}\leq n_Y$,
  \begin{flalign*}  \ \ \ \ \ \ \ \ m_{1,1}
          \leq\frac{2(m_1+n_Y+n^{{\rm even}}_{Y})+2}{3}+n_X-1 <\frac{2m_1}{3}+\frac{7n_Y}{3}.&&&\end{flalign*}
  It follows from (\ref{xeq2}) that
           \begin{flalign} \ \ \ \ \ \ \ \ \label{xinnew1} {m_1-m_{1,1}> m_1-(\frac{2m_1}{3}+\frac{7n_Y}{3})=\frac{m_1}{3}-\frac{7n_Y}{3}\geq \frac{7n_Y}{3}>2n_Y.}&&&\end{flalign}
         So
           \begin{flalign} \label{xeq9} \ \ \ \ \ \ \ \ m_{1,1}<m_1=|E_{G_1}|.&&&\end{flalign}
                  On the other hand, by (\ref{xeq8}) and (\ref{xeq6}),
           \begin{flalign}  \ \ \ \ \ \ \ \ \label{xeq10} m_{1,1}
    >{\lceil\frac{2(m_1+n_Y+n^{\rm even}_Y)}{3}\rceil-\frac{n_X}{2}\geq 9n_Y+7}.&&&\end{flalign} Then by (\ref{xeq4}),
      \begin{flalign} \label{xihu}  \ \ \ \ \ \ \ \ m_{1,1}-|E_{F_1}|=m_{1,1}-k_1
    >{9n_Y-n_X\geq 8n_Y}.&&&\end{flalign}
Now, by (\ref{xeq9}),   (\ref{xihu}) and the fact that $d_{G_0}(y)$ is even for each $y$ in $Y$, we can construct  an edge set $E_4$ of size $m_{1,1}$ in the following steps.  \begin{kst}
     \item[1)]  Let $E_{F_1}$ be included in $E_4$.

     \item[2)] \vspace{2mm} For each vertex $y$ in $Y$, if $d_{F_1}(y)$ is odd, then choose one  edge  in $E_{G_1}\setminus E_{F_1}$ that is incident to $y$ to $E_4$.

         \item[3)] \vspace{2mm} For each vertex $y$ in $Y$, if $|E_4(y)|\leq6$, then choose  $8-|E_4(y)|$ edges in $E_{G_1}\setminus E_4$ that are incident to $y$, and add them to $E_4$.

              \item[4)] \vspace{2mm} If $m_{1,1}$ is even, then choose $\frac{m_{1,1}}{2}-4n_Y$ $X$-links in $E_{G_1}\setminus E_4$, and add them to $E_4$. If   $m_{1,1}$ is odd, then choose $\frac{m_{1,1}-1}{2}-4n_Y$ $X$-links and one edge in $E_{G_1}\setminus E_4$,  and add them to $E_4$.
          \end{kst}
           That is, we have found an edge set $E_4$   of size  $m_{1,1}$ in $G_1$ that   includes  $E_{F_1}$ and satisfies the following two constrains:
    \begin{kst}
     \item[(C1)]  for   each $y$  in $Y$,    $|E_4(y)|\geq8$;

     \item[(C2)] \vspace{1mm}   for each $y$ but at most one (supposed to be $y'$ when exists) in $Y$,   $|E_4(y)|$ is even.
          \end{kst}
 See Figure \ref{fignewnew4} (1).

  \begin{figure}[htbp]
 \centering
 %\psfrag{a}{(a)}
 %\psfrag{b}{(b)}
 \includegraphics[height=4.1cm]{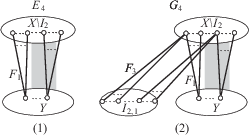}
\caption{\label{fignewnew4}   (1) The edge set $E_4\subseteq E_{G_1}$ of size $m_{1,1}= \lceil\frac{2m}{3}\rceil-m_{2,1}$ that contains $E_{F_1}$ and satisfies (C1) and (C2). (2)   The graph $G_4$ induced by $E_4\cup E_{F_2}$.}
\end{figure}

Let $G_4$ be the graph induced by  $E_4\cup E_{F_3}$;   see Figure \ref{fignewnew4} (2). Let $Y^{odd}_{G_4}$ be the odd-degree vertex set  within $Y$ in $G_4$. One has $Y^{odd}_{G_4}=\emptyset$ or $\{y'\}$ by our definition of $E_4$. The   claim below is to ensure that $G_4$ meets the conditions of    Lemma \ref{lem6}.

\begin{claim} \label{claim2}  There is a choice of $E_4$ such that the following  hold.  \begin{kst}\item[(i)] There is no Eulerian component in    $G_4$.
 \item[(ii)]\vspace{2mm}    If $I_{2,1}\cup Y^{odd}_{G_4}\neq\emptyset$, then there exists a component   in
$G_4$ that admits an odd-degree vertex in each of     $I_{2,1}\cup Y^{odd}_{G_4}$ and  $X\setminus I_2$.\end{kst}
\end{claim}

\begin{subproof}  Let $\mathcal H_1$ be the set of  Eulerian components of $G_4$. Let $\mathcal H_2$ be the set of    components of $G_4$   that     include  a vertex in $I_{2,1}\cup Y^{odd}_{G_4}$. Then $\mathcal H_1\cap \mathcal H_2=\emptyset$. Let $\mathcal H_3\subseteq\mathcal H_2$ be the set of components   that    admit   an odd-degree vertex in $X\setminus I_2$. We choose $E_4$ such that   $|\mathcal H_3|$ is maximum subject to that   $|\mathcal H_1|$ is minimum.

Recall that {$m_1-m_{1,1}>2n_Y$} from (\ref{xinnew1}). So
 there exists a
vertex $y^\ast$ in $Y$ with $|(E_{G_1}\setminus E_4)(y^\ast)|\geq3$.   Let $y^\ast x_1, y^\ast x_2, y^\ast x_3\in (E_{G_1}\setminus E_4)$. Suppose $H_0$, $H_1, H_2$ and   $H_3$ are the components  of $G_4$ that include $y^\ast$, $x_1,   x_2$ and  $x_3$, respectively.    For $\{i,j\}\subseteq\{1,2,3\}$, denote by $Q_{x_i,x_j}$ the connected graph induced by  $E_{H_0}\cup E_{H_i}\cup E_{H_j}\cup \{y^\ast x_i, y^\ast x_j\}$.
 \begin{kst}
 \item If      $x_i$ has  an even degree in $H_i$ for some   $i\in\{1,2,3\}$, then $x_i$ has an odd degree in $Q_{x_i,x_j}$ for each $j\in(\{1,2,3\}\setminus\{i\})$.

     \item \vspace{2mm} If $x_i$  has an odd degree in  $H_i$ (then there exists another odd-degree vertex $w_i$ in $H_i$) for each $i\in\{1,2,3\}$,   then

          \begin{kst}
          \item when   $H_j\neq H_k$   for some $j,k\in\{1,2,3\}$,      one has $w_j\neq x_k$, implying that $w_j$    has an   odd degree in $Q_{x_j,x_k}$;

               \item \vspace{2mm} when $H_1=H_2=H_3$,    $x_3$     has an   odd degree in $Q_{x_1,x_2}$.

                  \end{kst}
                  \end{kst}
                  So there is always a non-Eulerian   graph among  $Q_{x_1,x_2},Q_{x_1,x_3},Q_{x_2,x_3}$. Without loss of generality, suppose $Q_{x_1,x_2}$ is non-Eulerian.

\vspace{2mm}For (i), suppose to the contrary that $\mathcal H_1\neq \emptyset$ and let $H\in\mathcal H_1$. Denote by $X_H=V_H\cap I_1$ and  $Y_H=V_H\cap (Y\setminus Y^{odd}_{G_4})$ for short.  Let $yx\in E_H$ where $y\in Y_H$. Clearly, $yx$ is not a cut edge in $H$, since each Eulerian graph is 2-edge-connected.  If $d_{H}(y)< d_{G_1}(y)$,   let $yx'$ be an edge in
$G_1-E_4$ where  $x'$ is in some component $H'$  of $G_4$. Let ${\widehat{E}}_{4}=(E_4\setminus \{yx\})\cup \{yx'\}$  and     $\widehat{G}_{4}$ be the graph induced by $\widehat{E}_{4}\cup E_{F_3}$. Then $\widehat{E}_{4}$ satisfies (C1) and (C2); but $\widehat{G}_{4}$ has fewer Eulerian
components than $G_4$, since the graph induced by  $(E_{H}\setminus\{yx\})\cup E_{H'}\cup\{yx'\}$ is a non-Eulerian component in $\widehat{G}_{4}$,   a contradiction.
 So $d_{H}(y)=d_{G_1}(y)\geq14$. By Lemma \ref{lem4}, there exist  two edges   $yu_1,yu_2$ in $H$ such that $H-\{yu_1,yu_2\}$ is connected.
 Let $i_1\in\{1,2,3\}$ such that   $x_{i_1}\notin\{u_1,u_2\}$.
  \begin{kst}
  \item[1)]
  If none of $x_1,x_2,x_3$ is in $V_H$, let  $\widehat{E}_{4}=(E_4\setminus \{yu_1,yu_2\})\cup \{y^\ast x_1, y^\ast x_2\}$  and    $\widehat{G}_{4}$ be the graph induced by $\widehat{E}_{4}\cup E_{F_3}$. Then $\widehat{E}_{4}$ satisfies (C1) and (C2)  since $d_{\widehat{G}_{4}}(y)\geq14-2>8$; but  $H-\{yu_1,yu_2\}$ and $Q_{x_1,x_2}$ are two distinct non-Eulerian   components in  $\widehat{G}_{4}$, implying that $\widehat{G}_{4}$ has fewer Eulerian
components than $G_4$, a contradiction.
  \item[2)] \vspace{2mm} If one of  $x_1,x_2,x_3$ is in $V_H$, then let $i_2\in\{1,2,3\}\setminus\{i_1\}$ such that either $x_{i_1}$ or $x_{i_2}$ is in $V_H$. Let    $\widehat{E}_{4}=(E_4\setminus \{yu_1,yu_2\})\cup \{y^\ast x_{i_1}, y^\ast x_{i_2}\}$  and    $\widehat{G}_{4}$ be the graph induced by $\widehat{E}_{4}\cup E_{F_3}$. Then $\widehat{E}_{4}$ satisfies (C1) and (C2);  but   $Q_{x_{i_1},x_{i_2}}$  includes all edges in  $H$, and
  $Q_{x_{i_1},x_{i_2}}-\{yu_1,yu_2\}$   is  a non-Eulerian   component in $\widehat{G}_{4}$,   since at least one of $u_1, u_2$ (recall $x_{i_1}\notin\{u_1,u_2\}$ and so $\{x_{i_1},x_{i_2}\}\neq\{u_1,u_2\}$) has  an odd degree   in $Q_{x_{i_1},x_{i_2}}-\{yu_1,yu_2\}$, implying that $\widehat{G}_{4}$ has fewer Eulerian
components than $G_4$, also a contradiction.
\end{kst}
    Thus,  $\mathcal H_1=\emptyset$ and  (i) holds.

\vspace{2mm}For (ii), suppose to the contrary that $\mathcal H_3=\emptyset$. Let $H\in\mathcal H_2$, $X_H=V_H\cap (X\setminus I_2)$  and $Y_H=V_H\cap Y$.  Then $Y_H\neq \emptyset$, since $F_1\subseteq E_4$, $F_1$ covers  $X\setminus I_2$ and each edge in $F_2$ is incident to some vertex in $Y$. Let $yx\in E_H$  where $y\in Y_H$. By   assumption, $x$ has an even degree in $H$.

  If $d_{H}(y)< d_{G_1}(y)$, let $yx'$ be an edge in
$G_1-E_4$ where $x'$ is in some non-Eulerian component  $H'$ of $G_4$, as established by  (i). Let $\widehat{E}_{4}=(E_{4}\setminus \{yx\})\cup \{yx'\}$  and     $\widehat{G}_{4}$ be the graph induced by $\widehat{E}_{4}\cup E_{F_3}$. Then $\widehat{E}_{4}$ satisfies (C1) and (C2). Let    ${\widehat{H}}$ be the graph induced by   $E_H\cup E_{H'}\cup \{yx'\}\setminus\{yx\}$. Clearly,    $x$ has an odd degree   in ${\widehat{H}}$.
  \begin{kst}
  \item[1)] If $H'=H$ or  $yx$  is not a cut edge in $H$, then $\widehat{H}$ is a component  in   $\widehat{G}_{4}$ that   admits an odd-degree vertex in each of  $X\setminus I_2$ and $I_{2,1}\cup Y^{odd}_{G_4}$ (clearly, $\widehat{G}_{4}$ satisfies (i)), a contradiction.
\item[2)] \vspace{2mm} If $H'\neq H$ and  $yx$  is   a cut edge  in $H$, then $\widehat{H}$ consists of two components
 $\widehat{H}_1$ and $\widehat{H}_2$ that include  $x'$  and $x$, respectively.

  \begin{kst}
  \item
  Recall that $H'$ is non-Eulerian,   implying that   there exists an odd-degree vertex among $V_{H'}\setminus\{x'\}$ in  $H'$,   and this vertex    also has   an odd degree in $\widehat{H}_1$. So $\widehat{H}_1$ is non-Eulerian.

   \item  \vspace{2mm}
   On the other hand, $\widehat{H}_2$ is non-Eulerian since  $x$ has an odd degree in  $\widehat{H}_2$. Let  $u$ be an odd-degree vertex other than  $x$ in $\widehat{H}_2$. Clearly, $u$ also has an odd degree in $H$ since $d_{H}(u)=d_{\widehat{H}_2}(u)$, implying that $u$  is in $I_{2,1}\cup Y_0$ (since $\mathcal H_3=\emptyset$ by assumption).
    \end{kst}
    That is, $\widehat{G}_{4}$ satisfies (i), and  $\widehat{H}_2$   admits an odd-degree vertex in each of $X\setminus I_2$ and $I_{2,1}\cup Y^{odd}_{G_4}$,   a contradiction.
 \end{kst}
 So $d_{H}(y)=d_{G_1}(y)\geq14$.

\vspace{2mm}\noindent\textbf{Case 1.} If $y$ is incident to at least   three cut edges  $yu_1,yu_2,yu_3$    in $H$.

\vspace{2mm} In this case, let ${\widetilde{H}}_{u_i}$ be the component in $H-yu_i$ that includes $u_i$ for $i=1,2,3$. Clearly, $u_i$ has an odd degree in  $\widetilde{H}_{u_i}$ for $i=1,2,3$. Then there exists at least an odd-degree vertex in $\widetilde{H}_{u_i}$ other than
 $u_i$  that is in $I_{2,1}\cup Y^{odd}_{G_4}$ (since $\mathcal H_3=\emptyset$ by assumption) for each $i=1,2,3$.    This implies that, for $\{i,j\}\subseteq \{1,2,3\}$, $H-\{yu_i,yu_j\}$ is a graph consisting of three components $\widetilde{H}_{u_i}$,  $\widetilde{H}_{u_j}$ and $H-(V_{H_{u_i}}\cup V_{H_{u_j}})$, and each component   includes at least one vertex in $I_{2,1}\cup Y^{odd}_{G_4}$. Let $j_1,j_2\in\{1,2,3\}$ such that $\{u_{j_1},u_{j_2}\}\neq\{x_1,x_2\}$. Let $\widehat{E}_{4}=(E_4\setminus\{yu_{j_1},yu_{j_2}\})\cup \{y^\ast x_1, y^\ast x_2\}$  and    $\widehat{G}_{4}$ be the graph induced by $\widehat{E}_{4}\cup E_{F_3}$. Then $\widehat{E}_{4}$ satisfies (C1) and (C2) since $d_{\widehat{G}_{4}}(y)\geq14-2>8$. Let $\widehat{H}$ be the graph induced by  $E_{H-\{yu_{j_1},yu_{j_2}\}}\cup E_{Q_{x_1,x_2}}$.
\begin{kst}
  \item[1)]
   If neither  $x_1$ nor $x_2$ is   in $V_H$, then $\widehat{H}$ consists of four distinct non-Eulerian components    $Q_{x_1,x_2}$, $\widetilde{H}_{u_{j_1}}$,  $\widetilde{H}_{u_{j_2}}$ and $H-(V_{\widetilde{H}_{u_{j_1}}}\cup V_{\widetilde{H}_{u_{j_2}}})$ in $\widehat{G}_{4}$ (so $\widehat{G}_{4}$ satisfies (i)), where each of  $\widetilde{H}_{u_{j_1}},\widetilde{H}_{u_{j_2}}$ admits an odd-degree vertex in each of $X\setminus I_2$ and $I_{2,1}\cup Y^{odd}_{G_4}$, a  contradiction.
   \item[2)]\vspace{2mm} If at least one of  $x_1, x_2$ is   in $V_H$, then $Q_{x_1,x_2}$ contains all vertices  of at least one graph among $\widetilde{H}_{u_{j_2}}$,  $\widetilde{H}_{u_{j_2}}$ and $H-(V_{\widetilde{H}_{u_{j_1}}}\cup V_{\widetilde{H}_{u_{j_2}}})$. So      each component  of  $\widehat{H}$  includes  a vertex in $I_{2,1}\cup Y^{odd}_{G_4}$ (it follows that  $\widehat{G}_{4}$ satisfies (i)). What is more,   at least one of $u_{j_1}, u_{j_2}$ has  an odd degree in $\widehat{H}$ (recall that  $\{u_{j_1},u_{j_2}\}\neq\{x_1,x_2\}$),  also  a  contradiction.
\end{kst}

\vspace{2mm}\noindent\textbf{Case 2.}  If there exist at most two cut edges incident to $y$ in $H$.

\vspace{2mm} In this case, there  exist at least 12 non-cut edges    edges incident to $y$, since $d_{G_4}(y)\geq14$.
  By Lemma \ref{lem4},  there   exist two non-cut edges $yv_1,yv_2$ such that  $H-\{yv_1,yv_2\}$  is connected.   Clearly, $H-\{yv_1,yv_2\}$ admits an odd-degree vertex in $I_{2,1}\cup Y^{odd}_{G_4}$, and each of $v_1,v_2$ has an odd degree   in  $H-\{yv_1,yv_2\}$. Let $a_1\in \{1,2,3\}$ such that $x_{a_1}\notin\{v_1,v_2\}$.
        \begin{kst} \item[1)]
   If none of   $x_1,x_2,x_3$ is   in $V_H$, let $\widehat{E}_{4}=(E_4\setminus\{yv_1,yv_2\})\cup \{y^\ast x_1, y^\ast x_2\}$  and  $\widehat{G}_{4}$  be the graph induced by $\widehat{E}_{4}\cup E_{F_3}$. Then $\widehat{E}_{4}$ satisfies (C1) and (C2); but $Q_{x_1,x_2}$ and  $H-\{yv_1,yv_2\}$ are two distinct non-Eulerian components in $\widehat{G}_{4}$ (so  $\widehat{G}_{4}$ satisfies (i)),   where $H-\{yv_1,yv_2\}$ admits an odd-degree vertex in each of   $I_{2,1}\cup Y^{odd}_{G_4}$ and $X\setminus I_2$, a  contradiction.
    \item[2)] \vspace{2mm} If at least one of  $x_1,x_2,x_3$ is   in $V_H$, then let $a_2\in \{1,2,3\}\setminus\{a_1\}$ such that either $x_{a_1}$ or $x_{a_2}$ is in $V_H$. Let $\widehat{E}_{4}=(E_4\setminus\{yv_1,yv_2\})\cup \{y^\ast x_{a_1}, y^\ast x_{a_2}\}$  and   $\widehat{G}_{4}$ be the graph induced by $\widehat{E}_{4}\cup E_{F_3}$. Then $\widehat{E}_{4}$ satisfies (C1) and (C2).      Note that   $Q_{x_{a_1},x_{a_2}}$ includes all vertices in $H$, and  at least one of $v_1, v_2$ has  an odd degree in $Q_{x_{a_1},x_{a_2}}-\{yv_1,yv_2\}$ since $\{x_{a_1},x_{a_2}\}\neq\{v_1,v_2\}$. That is,   $\widehat{G}_{4}$ satisfies (i) and  contains a component with an   odd-degree vertex in each of $I_{2,1}\cup Y^{odd}_{G_4}$ and $X\setminus I_2$, a  contradiction.
\end{kst}
   Thus,  (ii) holds.

       This completes the proof of Claim \ref{claim2}.
\end{subproof}

\vspace{2mm} Next,   we   construct an antimagic labeling   on $G$, based on
the vertex partition and edge decomposition of $G$ above, and the careful partition  of  $[m]$.

\vspace{2mm} \noindent
\textbf{2. The careful  labeling assignment.}

\vspace{2mm}\noindent
\textbf{2.1. Top description of the labeling strategy.}

\vspace{2mm} We would construct a labeling on $G$ with $[m]$ ensuring that:
  \begin{itemize}
  \item[(i)] each vertex in $X$ receives a    $\{1,2\}$-vertex sum;
  \item[(ii)]  each vertex   but at most one in   $Y$ receives  a $0$-vertex sum;
    \item[(iii)]   the vertex sums in $X$ are pairwise different;
    \item[(iv)]   the  vertex sums in $Y$ are pairwise different;
           \item[(v)] the unique $\{1,2\}$-vertex sum (if exists) in $Y$   is  distinct from  any one in $X$.\end{itemize}

  For (i) and (ii), we label $E_3\cup E_4$ (that is, $E_{G_3}\cup E_{F_2}\cup E_{F_3}\cup E_4$) with $J$, where
  \begin{kst}
  \item we label  $G_3$ (if not empty) with   $1$-labels and 2-labels alternately along the Eulerian tour of each   component of $G_3$, so that each vertex  in $G_3$ receives a $0$-partial vertex  sums from $E_{G_3}$;

        \item\vspace{2mm} label $E_{F_2}$    such that the difference of the numbers of $1$-labels and $2$-labels on each    star component of $F_2$  is in $\{-2,-1,1,2\}$, so that all vertices in $I_2\setminus I_{2,1}$ receives a $\{1,2\}$-partial vertex sum from $E_{F_2}$;

               \item\vspace{2mm}label $G_4$ (that is, $E_{F_3}\cup E_4$)    in the way of Lemma \ref{lem6}, so that each vertex in $(X\setminus I_2)\cup I_{2,1}$ receives a $\{1,2\}$-partial vertex sum from $E_{G_4}$, and each vertex in $Y$ but at most one  receives a $0$-partial vertex sum   from $E_{G_4}$.
                   \end{kst}

  For (iii), we   finally  label $M$ with $n_X$   $0$-labels, such that  the label on $M(x)$   is   greater  than that on   $M(x')$ whenever $x$ receives a    partial vertex sums  not less than that $x'$ receives, right after  $G-M$ being labeled, for each vertex pair $x,x'$ in $X$.

   For (iv), we assign the labels carefully, such that     any pair of partial vertex sums    on $Y$ are either equal or have a difference of at least $3n_Y$, right after  $G-(M\cup E_1)$ being labeled, and   then  finally  label $E_1\cup M$ with $n_Y$   almost continuous  $0$-labels.

   For (v), we would do some     switching when necessary.

    The detailed procedure for label assignment is described as follows.

\vspace{2mm} \noindent
\textbf{2.2. The definitions of $m_{1,0}, m_{2,0}, s_1$ and $s_2$.}

 \vspace{2mm}For short, let
   $m_{1,0}\triangleq|E_{G_1}\setminus  E_{4}|,$ and $m_{2,0}\triangleq|E_{G[X]}\setminus E_3|$. One has the following relations:
\begin{flalign*}  \ \ \ \ \ \ \ \
    \left\{
   \begin{array}{ll}
   n_Y+n^{{\rm even}}_{Y}+m_{1}
   =(n_Y+n^{{\rm even}}_{Y}+m_{1,0})+m_{1,1}
   =l_{1,0}+(l_{1,1}+l_{1,2}),\vspace{1mm}\\ m_{2}=m_{2,0}+m_{2,1}
   =l_{2,0}+(l_{2,1}+l_{2,2}).
   \end{array}
  \right.
  &&&\end{flalign*}
 Then by (\ref{xeq7}), (\ref{xeq6}),
  \begin{flalign} \label{eq8} \ \ \ \ \ \ \ \
        m_{1,1}-(l_{1,1}+l_{1,2})=\gamma= (l_{2,1}+l_{2,2})-m_{2,1}
     =m_{2,0}-l_{2,0}.
   &&&\end{flalign}

   Let   $s_1$ and $s_2$  be the numbers of odd-size and  even-size  star components  in $F_2$, respectively.

    \vspace{2mm}\noindent
\textbf{2.3. The labeling steps.}

\vspace{2mm} For easy reading, we list the edge sets and their corresponding sizes in Table \ref{table1}, where we also list their assigned
   label sets (some of them will be defined more carefully later).

\begin{table}[htbp]
{\footnotesize
\begin{tabular}{|l|l|l|} \hline
edge set& size  & label set\\ \hline
  $E_{G_3}$  &$\varepsilon_1$ &$J_1$: the greatest ${\varepsilon_1}$ $\{1,2\}$-labels in $J$\\
    &  &\ \ \   (where   possibly    $E_{G_3}=\emptyset$ and   $\varepsilon_1=0$) \\
    &  &   \ \ \ see \textbf{Step 2}\\ \hline
   $E_{F_2}$&$k_2$ &$J_2$: some greatest    1-labels and   2-labels in $J\setminus J_1$  \\
    &  &\ \ \   (relying on the parities of  $s_1,s_2$, and the values of   $l_1$ and $l_2$)\\
    &  &   \ \ \ see \textbf{Step 3}\\ \hline
   $E_{F_3}$ &$k_3$&$J_3$: some greatest   1-labels and   2-labels in $J\setminus (J_1\cup J_2)$ \\
    &  &\ \ \   (subject to the labeling way in Lemma \ref{lem6})\\
    &  &   \ \ \ see \textbf{Step 4.1}\\
   \hline
   $E_4$ &$m_{1,1}$&$J_4=J\setminus (J_1\cup J_2\cup  J_3)=\{3i-2,3j-1|i\in[1,\theta_1],j\in[1,\theta_2]\}$\\
    &  &\ \ \     (where $|\theta_1-\theta_2|\leq2$ by the labeling way and the definition of $E_4$)  \\
    &  &   \ \ \ see \textbf{Steps 4.1, 4.2 and 4.3}\\
   \hline
    $E_{4,0}$ &0,1 or 2&$J_{4,0}=J_4\setminus \{3i-2,3i-1|i\in[1,\theta]\}$  \\
      & &\ \ \ (where $\theta\triangleq\frac{m_{1,1}}{2}$ when $\theta_1=\theta_2$, and $\theta\triangleq\lfloor\frac{m_{1,1}-1}{2}\rfloor$ when $\theta_1\neq\theta_2$)  \\
    &  &  \ \ \ see \textbf{Steps 4.2 and 4.3}\\
    \hline
   $E_{4,1}$ &$2n^{{\rm odd}}_{Y}$&$J_{4,1}=\{3i-2,3i-1|i\in[1,n^{{\rm odd}}_{Y}\}$\\
    &  &   \ \ \ see \textbf{Steps 4.2 and 4.3}\\ \hline
   $E_{4,2}$ &$2n^{{\rm even}}_{Y}$&$J_{4,2}=
    \{3(n^{{\rm odd}}_{Y}+\alpha)
    +3i-2,3(n^{{\rm odd}}_{Y}+\alpha)+3i-1|i\in[1,n^{{\rm even}}_{Y}]\}$ when $\alpha\geq0$\\
    &&\ \ \ \ \ \ or $\{3n^{{\rm odd}}_{Y}
    +3i-2,3n^{{\rm odd}}_{Y}+3i-1|i\in[1,n^{{\rm even}}_{Y}]\}$  when $\alpha<0$ \\
   && \ \ \  (where $\alpha\triangleq n^{{\rm even}}_{Y}+(l_{0}-m_{2,0}-\theta+1)=n^{{\rm even}}_{Y}+ l_{1,0}-\gamma-\theta+1$) \\
    &  &\ \ \ see \textbf{Steps 4.2 and 4.3}\\

    \hline
   $E_{4,3}$ &$2|\alpha|$&$J_{4,3}=\{3n^{{\rm odd}}_{Y}+3i-2,3n^{{\rm odd}}_{Y}+3i-1|i\in[1,\alpha]\}$ when $\alpha\geq0$\\
         & & \ \ \ \ \ \  or $\{3\theta -3i+1,3\theta-3i+2|i\in[1,-\alpha]\}$ when $\alpha<0$\\
    &  &\ \ \ see \textbf{Steps 4.2 and 4.3}\\
   \hline
   $E_{4,4}$ &$2(\theta-n_{2}-|\alpha|)$&$J_{4,4}=\{3i-2,3i-1|i\in[n_Y+\alpha+1,\theta]\}$ when $\alpha\geq0$\\
    &&\ \ \ \ \ \   or $\{3i-2,3i-1|i\in[n_Y+1,\theta+\alpha]\}$  when $\alpha<0$\\
    &  &\ \ \   see \textbf{Steps 4.2 and 4.3} \\  \hline
     $E_{G[X]}\setminus E_3$&$m_{2,0}$ &$O_1$: the greatest $m_{2,0}$ 0-labels in $O$\\
    &  &\ \ \ see \textbf{Step 1}\\\hline
   $E_2$&$n^{{\rm even}}_{Y}$ &$O_2=\{3i|i\in[1,n^{{\rm even}}_{Y}]\}$ when $n^{{\rm even}}_{Y}$ is odd\\
    &  &\ \ \   or $\{3i|i\in[1,n^{{\rm even}}_{Y}-1]\cup\{\frac{n^{{\rm even}}_{Y}}{2}\}\}$ when $n^{{\rm even}}_{Y}$ is even\\
    &  &\ \ \   see \textbf{Step 5} \\  \hline
       $E_{G_1}\setminus E_4$&$m_{1,0}$  &$O_3=\{3i|i\in[n_Y+n^{{\rm even}}_{Y}+1,l_0-m_{2,0}]\}$ \\
    &  &\ \ \ see \textbf{Step 6}\\  \hline
         $E_1$ &$n_Y-n_X$& $O_{4,1}$: the least $(n_Y-n_X)$ labels  in  $O\setminus (O_1\cup O_2\cup O_3)$\\
    &  &\ \ \ see \textbf{Step 7}\\  \hline
           $M$  &$n_X$    &$O_{4,2}=I\setminus (O_1\cup O_2\cup O_3\cup O_{4,1})$ \\
    &  &\ \ \ see \textbf{Steps 8 and 9}\\
\hline
\end{tabular}}
\vspace{2mm}
\caption{\label{table1}The edge sets in the decomposition of $E_G$ with their sizes, and their assigned label sets, where $E_{G_3}\cup E_{F_2}\cup E_{F_3}=E_3$, $E_4=\cup_{i=0}^4E_{4,i}.$}
\end{table}

\vspace{2mm}\noindent\textbf{Step 1.} Label $E_{G[X]}\setminus  E_{3}$.

\vspace{2mm}  Label $E_{G[X]}\setminus  E_{3}$ arbitrarily with   the $m_{2,0}$ greatest $0$-labels: \begin{flalign*}  \ \ \ \ \ \ \ \ O_1\triangleq\{3i|i\in [l_0-m_{2,0}+1,l_0]\}.&&&\end{flalign*}
Clearly, each vertex in $X$ receives a $0$-partial vertex sum in this step.

 \vspace{2mm}\noindent  \textbf{Step 2.}  Label $G_{3}$.

   \vspace{2mm}   If   $G_{3}$ is not empty, define the label set $J_1$ as the  $\varepsilon_1$ greatest $\{1,2\}$-labels:
   \begin{flalign*}  \ \ \ \ \ \ \ \ J_1\triangleq\{3(l_1-i)+1,3(l_2-i)+2|i\in[1,\frac{\varepsilon_1}{2}]\}.
 &&&\end{flalign*}
  Label  each component  of   $G_{3}$  along the Eulerian tour   using the   $1$-labels  and the $2$-labels  alternately  in $J_1$.

   Clearly,  each vertex in $X$ still receives a $0$-partial vertex sum to this step.

 \vspace{2mm}\noindent \textbf{Step 3.}  Label $F_2$.

  \vspace{2mm} Define integer set $J_2$ as follows.  If $s_1>0$    or  $s_2$ is even, then let
     \begin{flalign*}  \ \ \ \ \ \ \ \
    J_2\triangleq\left\{
   \begin{array}{ll}
   \{3(l_1-\frac{\varepsilon_1}{2}-i)+1,3(l_2-\frac{\varepsilon_1}{2}-j)+2|
  i\in[1,\lceil\frac{k_2}{2}\rceil],
  j\in[1,\lfloor\frac{k_2}{2}\rfloor]\},&    {\rm if}\  l_1=l_2+1, \vspace{1mm}  \\
   \{3(l_1-\frac{\varepsilon_1}{2}-i)+1,3(l_2-\frac{\varepsilon_1}{2}-j)+2|
  i\in[1,\lfloor\frac{k_2}{2}\rfloor],
    j\in[1,\lceil\frac{k_2}{2}\rceil]\},&   {\rm if}\    l_1=l_2.
   \end{array}
  \right. &&&
  \end{flalign*}
     If    $s_1=0$     and $s_2$ is odd, then let
        \begin{flalign*} \ \ \ \ \ \ \ \
    J_2\triangleq\left\{
   \begin{array}{ll}
   \{3(l_1-\frac{\varepsilon_1}{2}-i)+1,3(l_2-\frac{\varepsilon_1}{2}-j)+2|
  i\in[1,\frac{k_2+2}{2}],j\in[1,\frac{k_2-2}{2}]\},&   {\rm if}\   l_1=l_2+1, \vspace{1mm}  \\
   \{3(l_1-\frac{\varepsilon_1}{2}-i)+1,3(l_2-\frac{\varepsilon_1}{2}-j)+2|
  i\in[1,\frac{k_2-2}{2}],j\in[1,\frac{k_2+2}{2}]\},&  {\rm if}\     l_1=l_2.
   \end{array}
  \right.&&&
  \end{flalign*}

 Let  $\{H_i|i\in[1,s_2+s_1]\}$ be the set of star components of $F_2$, where $H_i$ has an even size   if $i\leq s_2$ and   has an odd size  if $i>s_2$.
 Label $H_1,H_2,\ldots,H_{s_2+s_1}$ in the order. For $i\in[1,s_2+s_1]$, suppose the greatest unused $\{1,2\}$-labels is a ${\mu_i}$-label ($\mu_i=1,2$) before we label $H_i$.
  \begin{kst}
  \item[1)] For $i\in[1, s_2]$, assign both the greatest $\frac{{|E_{H_i}|}}{2}+1$ unused ${\mu_i}$-labels and the greatest   $\frac{{|E_{H_i}|}}{2}-1$ unused  $(3-\mu_i)$-labels to  $E_{H_i}$.
          \item[2)] \vspace{2mm}
          For $i\in[s_2+1,s_2+s_1]$,   assign both   the greatest $\frac{{|E_{H_i}|+1}}{2}$ unused
${\mu_i}$-label and $\frac{{|E_{H_i}|-1}}{2}$ unused
$({3-\mu_i})$-labels to  $E_{H_i}$.
\end{kst}

The labels used here are exactly  that in $J_2$, and  each vertex in $I_2\setminus I_{2,1}$ receives a $\{1,2\}$-partial sum to this step.

\vspace{2mm}\noindent\textbf{Step 4.}   Label $E_{G_4}$ (that is, $E_4\cup E_{F_3}$).

\vspace{2mm}\noindent\textbf{Step 4.1.}   Rough label assignment of $J\setminus (J_1\cup J_2)$ to $E_{G_4}$,  according to  Lemma \ref{lem6}.

 \vspace{2mm} Clearly, $|J\setminus (J_1\cup J_2)|=|E_{G_4}|=|E_{4}\cup E_{F_3}|=m_{1,1}+k_3$. The   claim below is to ensure that the label set $J\setminus (J_1\cup J_2)$ satisfies  the conditions of    Lemma \ref{lem6}.

\begin{claim} \label{claim3} The difference of the numbers of 1-labels and 2-labels in  $J\setminus (J_1\cup J_2)$ is at most 2.
\end{claim}
\begin{subproof} Note that $\varepsilon_1$  is even since  $G_3$ is an even bipartite graph.

\vspace{2mm}\noindent\textbf{Case 1.} If  $I_{2,1}=\emptyset$ and   $m_{1,1}$ is even.

 \vspace{2mm} In this case,
  $k_3=|I_{2,1}|=0$ and $|J\setminus (J_1\cup J_2)|=m_{1,1}$ is even.  Then  $l_1+l_2$ and $k_2$ have the same parity, since  $l_1+l_2=\lceil\frac{2m}{3}\rceil=m_{1,1}+m_{2,1}=m_{1,1}+\varepsilon_1
  +k_2+k_3$ by (\ref{xeq3})   and the definition  of $E_4$. By the definitions of $J_1,J_2$, when   $s_1>0$     {\rm or}   $s_2$  {\rm is}\ {\rm even}, one has
 \begin{flalign*}\ \ \ \ \ \ \ \ J\setminus (J_1\cup J_2)=
   \{3i-2,3i-1|
  i\in[1, \frac{m_{1,1}}{2}]\};&&&  \end{flalign*}
      when $s_1=0$     {\rm and}   $s_2$  {\rm is}\ {\rm odd} (then $k_2$ is even, and so      $l_1+l_2$ is even, implying that $l_1=l_2$), one has
\begin{flalign*}\ \ \ \ \ \ \ \ J\setminus (J_1\cup J_2)=\{3i-2,3j-1|
  i\in[1,\frac{m_{1,1}+2}{2}],j\in[1,\frac{m_{1,1}-2}{2}]\}.&&&
\end{flalign*}

\vspace{2mm}\noindent\textbf{Case 2.} If  $I_{2,1}\neq\emptyset$ or   $m_{1,1}$ is odd.

    \vspace{2mm} In this case,   by the definitions of $J_1,J_2$,  when    $s_1>0$     {\rm or}   $s_2$  {\rm is}\ {\rm even},
\begin{flalign*}\ \ \ \ \ \ \ \ J\setminus (J_1\cup J_2)=
   \{3i-2,3j-1|
  i\in[1, \lceil\frac{m_{1,1}+k_3}{2}\rceil],j\in[1, \lfloor\frac{m_{1,1}+k_3}{2}\rfloor]\};&&&\end{flalign*}
   when $s_1=0$     {\rm and}   $s_2$  {\rm is}\ {\rm odd} (then $m_{1,1}+k_3$ and $l_{1}+l_2$ have the same parity since $l_1+l_2=m_{1,1}+\varepsilon_1+k_2+k_3$),
          \begin{flalign*}\ \ \ \ \ \ \ \
    J\setminus (J_1\cup J_2)=\left\{
   \begin{array}{ll}
  \{3i-2,3j-1|
  i\in[1, \frac{m_{1,1}+k_3-1}{2}],j\in[1, \frac{m_{1,1}+k_3+1}{2} ]\},&   {\rm if}\   l_1=l_2+1, \vspace{1mm}  \\
    \{3i-2,3j-1|
  i\in[1, \frac{m_{1,1}+k_3+2}{2}],j\in[1, \frac{m_{1,1}+k_3-2}{2}]\},&    {\rm if}\  l_1=l_2.
   \end{array}
  \right.
  &&&\end{flalign*}

  Thus, the difference of the numbers of 1-labels and 2-labels in  $J\setminus (J_1\cup J_2)$ is at most 2, in both cases.

   This completes the proof of Claim \ref{claim3}.
\end{subproof}

  \vspace{2mm}Label $E_{G_4}$ (that is, $E_4\cup E_{F_3}$) with the $\{1,2\}$-label set $J\setminus (J_1\cup J_2)$ in the way of Lemma \ref{lem6},  such that:
 \begin{kst}
 \item[(i)]  each vertex in  $(X\setminus I_2)\cup I_{2,1}$ receives  a $\{1,2\}$-partial sum;
    \item[(ii)]  \vspace{2mm}
  if $I_{2,1}=\emptyset$ and $m_{1,1}$ is even (that is,   there exists no odd-degree vertex in $I_{2,1}\cup Y$ in $G_4$), then $l_{1,y}=l_{2,y}$ for each vertex $y$ in $Y$ but at most one (denoted $y'$ if it exists)   and   $|l_{1,y'}-l_{2,y'}|=2$, where $l_{\mu,v}$ ($\mu=1,2$) is the number of $\mu$-labels assigned to $E_{G_4}(v)$ from $J\setminus (J_1\cup J_2)$ for each $v\in (X\setminus I_2)\cup I_{2,1}\cup Y$;  if $I_{2,1}\neq\emptyset$ or $m_{1,1}$ is odd (that is, there exists at least one odd-degree vertex in $I_{2,1}\cup Y$ in $G_4$), then $|l_{1,v}-l_{2,v}|\leq 1$ for each vertex $v$  in $I_{2,1}\cup Y$;
   \item[(iii)]\vspace{2mm} there exist  at least one 1-partial vertex sum and one  2-partial vertex sum  in $(X\setminus I_2)\cup I_{2,1}$ (since there exists at most one odd-degree vertex in $Y$ in $G_4$, each vertex in $Y$ has degree at least eight in $G_4$, and $n_Y\geq\frac{|V_G|}{2}$).
       \end{kst}
\vspace{2mm}

 Let $\rho_{3,\mu}$ ($\mu=1,2$) be the number of $\mu$-labels   assigned to $E_{F_3}$, subject to the labeling way in Lemma \ref{lem6}. Let    $J_3$ be the label set consisting of the greatest $\rho_{3,1}$ $1$-labels and the greatest $\rho_{3,2}$ $2$-labels
 in $J\setminus (J_1\cup J_2)$.

 Subject to the labeling way in Lemma \ref{lem6},    assign the labels in $J_3$ to $E_{F_3}$ arbitrarily.

 Let $J_4\triangleq J\setminus(J_1\cup J_2\cup J_3)$.   We would assign $J_4$ to $E_4$   carefully, after defining  some special subsets of $E_4$ and subsets of $J_4$.

  \vspace{2mm}\noindent\textbf{Step 4.2.}   Defining $\theta_1,\theta_2,\theta,$   $E_{4,i}$ ($i\in[0,4]$), $J_{4,0}$ and $\alpha$.

   \vspace{2mm}Let $\theta_\mu$ be the number of $\mu$-labels in $J_4$ for $\mu=1,2$. That is,
 \begin{flalign*}    \ \ \ \ \ \ \ \  J_4=\{3i-2,3j-1|i\in[1,\theta_1],j\in[1,\theta_2]\}. &&&
\end{flalign*}
Then  $|\theta_1-\theta_2|\leq2$ by Step 4.1 (ii) and by  the definition of $E_4$. Let $\theta\triangleq\min\{\theta_1,\theta_2\}$.
  That is,
 \begin{flalign} \label{eq9}  \ \ \ \ \ \ \ \  \theta&=\left\{ \begin{array}{ll} \frac{m_{1,1}}{2}, &  {\rm if}\  \theta_1=\theta_2,\\
  \lfloor\frac{m_{1,1}-1}{2}\rfloor,& {\rm if}\   \theta_1\neq\theta_2.\\
  \end{array}\right. &&&
\end{flalign}
 By (\ref{xeq10}) and (\ref{eq9}),
  \begin{flalign}  \ \ \ \ \ \ \ \  \label{yeq1} \theta\geq
  \frac{m_{1,1}-2}{2}\geq \frac{9}{2}n_Y. &&&
\end{flalign}   When $|\theta_1-\theta_2|=1$ (then there exists a  unique odd-degree vertex  $y'$ in $Y$ in $G_4$), let      $e'$ be an edge incident to $y'$ in $E_{4}$;  when $|\theta_1-\theta_2|=2$ (then there exists a unique  vertex  $y'$ in $Y$ in $G_4$ with $|l_{1,y'}-l_{2,y'}|=2$), let $e', e''$ be the two edges of some $X$-link in $E_{4}$ that are incident to   $y'$. Let
    \begin{flalign*} \ \ \ \ \ \ \ \
    Y_0\triangleq\left\{
   \begin{array}{ll}
  \emptyset,&     {\rm if}\   \theta_1=\theta_2,  \vspace{1mm}  \\
   \{y'\},&   {\rm if}\ \theta_1\neq\theta_2,  \vspace{1mm}  \\
    \end{array}
  \right. \ \ \ \   E_{4,0}\triangleq\left\{
   \begin{array}{ll}
  \emptyset,&      {\rm if}\  \theta_1=\theta_2,  \vspace{1mm}  \\
   \{e'\},&  {\rm if}\  |\theta_1-\theta_2|=1,  \vspace{1mm}  \\
  \{e',e''\},&  {\rm if}\ |\theta_1-\theta_2|=2,\vspace{1mm}\\
   \end{array}
  \right.   &&&
\end{flalign*}
 and let \begin{flalign} \label{j40}    \ \ \ \ \ \ \ \  J_{4,0}&\triangleq J_4\setminus \{3i-2,3i-1|i\in[1,\theta]\}. &&&
\end{flalign}
Suppose $Y=\{y_i|i\in[1,n_Y]\}$ where $d_G(y_i)\leq d_G(y_j)$ whenever $i<j$. Let $e_{i,1}, e_{i,2}, e_{i,3}, e_{i,4},e_{i,5}, e_{i,6}\in   E_{4}(y_i)$ for each  $i\in [1,n_Y]$, where $e_{i,j}$ receives a 1-label for $j=1,3,5$ and $e_{i,j}$ receives a 2-label for $j=2,4,6$ under the labeling way in Lemma \ref{lem6}.  Let  $E_{4,1}=\{e_{i,1}, e_{i,2}| y_i\in Y_{\rm odd}\}$, $E_{4,2}=\{e_{i,1}, e_{i,2}| y_i\in Y_{\rm even}\}$. 
Denote by
  \begin{flalign}\label{eq13} \ \ \ \ \ \ \ \  \alpha\triangleq n^{\rm even}_Y+(l_{0}-m_{2,0}-\theta+1)=n^{\rm even}_Y+ l_{1,0}-\gamma-\theta+1, &&&
\end{flalign}
where the second equality holds by (\ref{eq8}) and by  the fact that $l_0=l_{1,0}+l_{2,0}$.
\begin{claim} \label{claim4} One has that  $|\frac{\alpha}{2}|< n_Y$.
\end{claim}
\begin{subproof} By   (\ref{eq8}), (\ref{eq9}) and (\ref{eq13}),
\begin{flalign}  \label{eq001}  \ \ \ \ \ \ \ \  \begin{split} \alpha
 &=n^{\rm even}_Y+ l_{1,0}-\theta+1-\gamma\\
&\geq n^{\rm even}_Y+ l_{1,0}-\frac{l_{1,1}+l_{1,2}+\gamma}{2} +1-\gamma \\
&=  (n^{\rm even}_Y+1)+(l_{1,0}-\frac{l_{1,1}+l_{1,2}}{2})-\frac{3\gamma}{2}
\\
&\geq n^{\rm even}_Y-\frac{3\gamma}{2},\end{split}&&&
\end{flalign}  since
\begin{flalign*}  \ \ \ \ \ \ \ \  -1\leq l_{1,0}-\frac{l_{1,1}+l_{1,2}}{2}\leq0;&&
\end{flalign*}
 and one has that
\begin{flalign}  \label{eq002}  \ \ \ \ \ \ \ \   \begin{split}\alpha
 &=n^{\rm even}_Y+ l_{1,0}-\theta+1-\gamma\\
&\leq n^{\rm even}_Y+ l_{1,0}  -\lfloor\frac{l_{1,1}+l_{1,2}+\gamma-1}{2}\rfloor +1-\gamma \\
& \leq n^{\rm even}_Y+ l_{1,0}   -\frac{l_{1,1}+l_{1,2}+\gamma-2}{2} +1-\gamma\\
&\leq n^{\rm even}_Y-\frac{3\gamma}{2}+2.\end{split}&&
\end{flalign}
 Recall that $\frac{-n_X}{2}<\gamma<n_X$  by (\ref{xeq8}). Then by (\ref{eq001}),  (\ref{eq002}), and by the fact that  $\max\{n^{\rm even}_Y,n_X\}\leq n_Y$ and $n_Y\geq15$,
  \begin{kst}
  \item
  if $0\leq\gamma<n_X$,
   \begin{flalign*}\ \ \ \ \ \ \ \ \ \ \ \ \ \ \ \ |\frac{\alpha}{2}|\leq \max\{\frac{n^{\rm even}_Y+2}{2},\frac{3\gamma}{4}\}
\leq\max\{\frac{n^{\rm even}_Y+2}{2},\frac{3(n_X-1)}{4}\}<n_Y;&&&
\end{flalign*}
 \item \vspace{2mm}if $-\frac{n_X}{2}<\gamma<0$,
 \begin{flalign*}\ \ \ \ \ \ \ \ \ \ \ \ \ \ \ \ \frac{\alpha}{2}\leq \frac{1}{2}({n^{\rm even}_Y}-\frac{3\gamma}{2}+2)\leq
\frac{n^{\rm even}_Y}{2}+\frac{3n_X}{8}+1<n_Y.&&&
\end{flalign*}
\end{kst}
 Thus,   $|\frac{\alpha}{2}|< n_Y$ in both cases.

 This completes the proof of Claim \ref{claim4}.
\end{subproof}

\vspace{2mm}Define  \begin{flalign*}\ \ \ \ \ \ \ \
    E_{4,3}\triangleq\left\{
   \begin{array}{ll}
  \{e_{i,3}, e_{i,4},e_{j,5}, e_{j,6}| i\in[1,\lfloor\frac{\alpha}{2}\rfloor], j\in[1,\lceil\frac{\alpha}{2}\rceil]\},&      {\rm if}\ \alpha\geq0,  \vspace{1mm}  \\
  \{e_{i,3}, e_{i,4},e_{j,5}, e_{j,6}| i\in[n_Y+\lfloor\frac{\alpha}{2}\rfloor+1,n_Y], j\in[n_Y+\lceil\frac{\alpha}{2}\rceil+1,n_Y]\},&  {\rm if}\  \alpha<0.  \vspace{1mm}
   \end{array}
  \right.   &&&\end{flalign*}
  Note that $E_{4,3}$ is well-defined by Claim \ref{claim4}. Let $E_{4,4}\triangleq E_{4}\setminus (E_{4,0}\cup E_{4,1}\cup E_{4,2}\cup E_{4,3})$.

\vspace{2mm}\noindent\textbf{Step 4.3.}   Careful  label assignment of $J_4$ to $E_{4}$.

\vspace{2mm} Now we have  the following     label assignment on $E_{4}$.
 { \begin{kst}
 \item[1)]  Assign   the   labels in $J_{4,0}$ arbitrarily  to
$E_{4,0}$.

 \item[2)]\vspace{2mm}   Pair the labels in $J_{4,1}\triangleq\{3i-2,3i-1|i\in[1,n^{\rm odd}_Y\}$ such that each pair has the sum  $3n^{\rm odd}_Y$ and is assigned (properly subject to the labeling way of Lemma \ref{lem6}) to some $X$-link in $E_{4,1}$. Then for each   $y$  in $Y_{{\rm odd}}$, \begin{flalign} \ \ \ \ \ \ \ \ \ \ \ \ \label{eq011}  \sigma^{E_{4,1}}(y)=p_1 \triangleq3n^{\rm odd}_Y.&&&
\end{flalign}
\item[3)] \vspace{2mm} Let $
    J_{4,2}\triangleq
    \{3(n^{\rm odd}_Y+\max\{0,\alpha\})+3i-2,3(n^{\rm odd}_Y
    +\max\{0,\alpha\})+3i-1|i\in[1,n^{\rm even}_Y]\}$.
    \begin{kst}
    \item
 When  $n^{\rm even}_Y$ is odd,  pair the labels in $J_{4,2}$ in the way of Lemma \ref{lem3} (i), and assign the label pairs to the $X$-links in   $E_{4,2}$,
  such that $Y_{{\rm even}}$ receives continuous partial vertex sums from $E_{4,2}$. That is,
  \begin{flalign}\ \ \ \ \ \ \ \ \ \ \ \ \ \ \ \ \ \ \ \label{xin1}   \sigma^{E_{4,2}}(Y_{{\rm even}})= \{6(n^{\rm odd}_Y+\max\{0,\alpha\})+\frac{3(n^{\rm even}_Y-1)}{2}+3i|i\in[1,n^{\rm even}_Y]\}. &&& \end{flalign}
\item \vspace{2mm} When $n^{\rm even}_Y$ is even,   pair the labels in $J_{4,2}$ in the way of Lemma \ref{lem3} (ii), and assign the label pairs to the $X$-links in   $E_{4,2}$,
  such that
  \begin{flalign}\ \ \ \ \ \ \ \ \ \ \ \ \ \ \ \ \ \ \  \label{xin2} \begin{split}  \sigma^{E_{4,2}}(Y_{{\rm even}})=&  \{6(n^{\rm odd}_Y+\max\{0,\alpha\})+\frac{3 n^{\rm even}_Y}{2}+3+3i|i\in[1,n^{\rm even}_Y-1]\}\\
   &\cup \{6(n^{\rm odd}_Y+\max\{0,\alpha\})+3\},\end{split}&&& \end{flalign} with $n_Y-1$  continuous 0-partial vertex sums appearing in this collection.

\end{kst}

\item[4)] \vspace{2mm}  Let
\begin{flalign*} \ \ \ \ \ \ \ \ \ \ \ \
    J_{4,3}\triangleq\left\{
   \begin{array}{ll}
  \{3n^{\rm odd}_Y+3i-2,3n^{\rm odd}_Y+3i-1|i\in[1,\alpha]\},&   {\rm if}\    \alpha\geq0,  \vspace{1mm}  \\
   \{3\theta -3i+1,3\theta-3i+2|i\in[1,-\alpha]\},&   {\rm if}\     \alpha<0.  \\
   \end{array}
  \right.   &&&
\end{flalign*} Label $E_{4,3}$ with
  $J_{4,3}$, where the labels
  are paired such that each pair has the sum
  \begin{flalign} \ \ \ \ \ \ \ \ \ \ \ \  \label{eq014}
    p_3\triangleq\left\{
   \begin{array}{ll}
  6n^{\rm odd}_Y+3\alpha,&      {\rm if}\   \alpha\geq0,  \vspace{1mm}  \\
   6\theta+3\alpha,&  {\rm if}\ \alpha<0,     \vspace{1mm}
   \end{array}
  \right.   &&&
\end{flalign}   and is assigned to some
  $X$-link in $E_{4,3}$. Then
  \begin{kst}
    \item when $\alpha$ is even, one has
  \begin{flalign} \ \ \ \ \ \ \ \ \ \ \ \ \ \ \ \ \ \ \ \ \ \ \ \
    \sigma^{E_{4,3}}(y_i)=2p_3,&&&
\end{flalign}
    for each $i$ in $[1,\frac{\alpha}{2}]$ (resp. $[n_Y+\frac{\alpha}{2}+1,n_Y]$) if $\alpha\geq0$ (resp. $\alpha<0$);

       \item \vspace{2mm}when $\alpha$ is odd, one has \begin{flalign} \ \ \ \ \ \ \ \ \ \ \ \ \ \ \ \ \ \ \ \ \ \ \ \
    \sigma^{E_{4,3}}(y_i)=2p_3,&&&
\end{flalign}   for each $i$ in $[1,\frac{\alpha-1}{2}]$ ({resp. $[n_Y+\frac{\alpha+1}{2}+1,n_Y]$}) if $\alpha\geq0$ (resp. $\alpha<0$), and
    \begin{flalign} \ \ \ \ \ \ \ \ \ \ \ \ \ \ \ \ \ \ \ \ \ \ \ \
    \sigma^{E_{4,3}}(y_{\frac{\alpha+1}{2}})=p_3\ ({\rm resp.\ } \sigma^{E_{4,3}}(y_{n_Y+\frac{\alpha+1}{2}})=p_3).&&&
\end{flalign}
\end{kst}
\item[5)] \vspace{2mm}
 Label $E_{4,4}$ with
   \begin{flalign*} \ \ \ \ \ \ \ \ \ \ \ \
    J_{4,4}\triangleq\left\{
   \begin{array}{ll}
  \{3i-2,3i-1|i\in[n_Y+\alpha+1,\theta]\},&     {\rm if}\    \alpha\geq0,  \vspace{1mm}  \\
   \{3i-2,3i-1|i\in[n_Y+1,\theta+\alpha]\},&    {\rm if}\    \alpha<0,  \vspace{1mm}  \\
   \end{array}
  \right.   &&&
\end{flalign*}
     where  the labels are paired such that each pair   has   the sum
      \begin{flalign} \ \ \ \ \ \ \ \ \ \ \ \ \label{eq00001}
  p_4\triangleq3(n_Y+\alpha+\theta),   &&&
\end{flalign}
  and is assigned to some $X$-link
   in $E_{4,4}$.

  \end{kst}

   \noindent \vspace{2mm} \textbf{Step 5.} Label $E_2$, according to the label assignment to $E_{4,2}$.

  \vspace{2mm} Let
    \begin{flalign} \label{I2} \ \ \ \ \ \ \ \ \ \ \ \
    O_{2}\triangleq\left\{
   \begin{array}{ll}
  \{3i|i\in[1,n^{\rm even}_Y]\},&   {\rm when}\    n^{\rm even}_Y\ {\rm is\ odd},\   \vspace{1mm}  \\
   \{3i|i\in[1,n^{\rm even}_Y-1]\cup\{\frac{3n^{\rm even}_Y}{2}\}\},&   {\rm when}\    n^{\rm even}_Y\ {\rm is\ even}.  \\
   \end{array}
  \right.   &&&
\end{flalign}

\begin{kst}
\item[1)]
When $n^{\rm even}_Y$ is odd, by (\ref{xin1}) and  (\ref{I2}),   label $E_{2}$ with the continuous  $0$-labels  in $O_2$,
  such that
   \begin{flalign} \ \ \ \ \ \ \ \ \ \ \ \  \label{eq012} \begin{split}\sigma^{E_{4,2}\cup E_{2}}(y)=p_2&\triangleq 6(n^{\rm odd}_Y+\max\{0,\alpha\})+\frac{9n^{\rm even}_Y}{2}+\frac{3}{2}
   \\
   &=\frac{3(3n_Y+n^{\rm odd}_Y+1)}{2}+6\max\{0,\alpha\}\end{split} &&& \end{flalign} for each $y$ in $Y_{{\rm even}}$.

 \item[2)]\vspace{2mm}  When  $n^{\rm even}_Y$ is even, by (\ref{xin2}) and  (\ref{I2}),   label $E_{2}$ with $O_2$ (containing $n_Y-1$  continuous 0-labels and the other one label $\frac{3n^{\rm even}_Y}{2}$),
  such that all vertices in $Y$ receive precisely  a common partial vertex sum from $E_{4,2}\cup E_{2}$. That is,
   \begin{flalign} \ \ \ \ \ \ \ \ \ \ \ \ \label{eq013}   \sigma^{E_{4,2}\cup E_{2}}(y)=p_2'\triangleq6(n^{\rm odd}_Y+\max\{0,\alpha\})+\frac{9n^{\rm even}_Y}{2}+3 =p_2+\frac{3}{2}, &&& \end{flalign} for each $y$ in $Y_{{\rm even}}$.
\end{kst}
 \vspace{2mm} \noindent\textbf{Step 6.} Label $E_{G_1}\setminus E_4$.

\vspace{2mm} Let \begin{flalign*} \ \ \ \ \ \ \ \ \ \ \ \ O_3\triangleq\{3i|i\in[n_Y+n^{\rm even}_Y+1,l_0-m_{2,0}]\}.  &&&
\end{flalign*}  Pair the labels in $O_3$ such that each pair has precisely the sum
 \begin{flalign} \ \ \ \ \ \ \ \ \ \ \ \ \label{eq018}
 3(n_Y+n^{\rm even}_Y+l_{0}-m_{2,0}+1)=3(n_Y+\alpha+\theta) = p_4,
 &&&\end{flalign} by (\ref{eq13}) and (\ref{eq00001}). Assign the pairs to   $X$-links in $E_{G_1}\setminus E_4$.

  \vspace{2mm}     Let $\sigma_1(y)$ be the partial vertex sum of  $v$  for each   $v\in V_G$ right after Step 6.
        \begin{claim} \label{claim5}Any pair of partial vertex sums in $Y\setminus Y_0$ from $E\setminus (E_1\cup M)$ are either equal or have a difference greater than
 $3n_Y$. \end{claim}

   \begin{subproof}          By   (\ref{eq012}), (\ref{eq013}) and Claim \ref{claim4},
 \begin{flalign} \ \ \ \ \ \ \ \ \ \ \ \ \label{zeq1}
      \frac{3(3n_Y+n^{\rm odd}_Y+1)}{2}\leq  p_2<p_2'\leq \frac{3(3n_Y+n^{\rm odd}_Y+2)}{2}+6\max\{0,\alpha\}.
     &&&
\end{flalign} By   (\ref{eq018}) and (\ref{yeq1}),
 \begin{flalign} \ \ \ \ \ \ \ \ \ \ \ \ \label{zeq2}
      p_4=3(n_Y+\alpha+\theta)\geq 3(n_Y+\alpha+\frac{9n_Y}{2})=\frac{33 n_Y}{2}+3\alpha.
     &&&
\end{flalign}
Then
\begin{kst}
\item[1)]  By  (\ref{eq011}) and (\ref{zeq1}),
 \begin{flalign} \ \ \ \ \ \ \ \ \ \ \ \ \label{aeq1}
       \min\{p_2'-p_1, p_2-p_{1}\}
       \geq \frac{3(3n_Y+n^{\rm odd}_Y+1)}{2}-3n^{\rm odd}_Y>3n_Y.
     &&&
\end{flalign}
\item[2)] \vspace{2mm}   By (\ref{zeq1}), (\ref{zeq2}) and Claim \ref{claim4},
 \begin{flalign} \ \ \ \ \ \ \ \ \ \ \ \ \label{aeq2}
    \begin{split}   \min\{p_4-p_{2},p_4-p'_{2}\}&\geq
       (\frac{33n_Y}{2}+3\alpha)-[\frac{3(3n_Y+n^{\rm odd}_Y+2)}{2}+6\max\{0,\alpha\}]\\
       &\geq \frac{21n_Y}{2}-3|\alpha|-3\\
        &\geq \frac{21n_Y}{2}-6n_Y-3\\
         &> 4n_Y.
  \end{split}  &&&
\end{flalign}
\item[3)] \vspace{2mm}
When $\alpha\geq0$, by  (\ref{eq014}) and (\ref{zeq2}),
\begin{flalign*} \ \ \ \ \ \ \ \ \ \ \ \
       p_4-p_3\geq (\frac{33 n_Y}{2}+3\alpha)-(6n^{\rm odd}_Y+3\alpha)>10n_Y;&&&
\end{flalign*}
when $\alpha<0$, by  (\ref{eq014}), (\ref{eq00001}) and (\ref{yeq1}),
\begin{flalign*} \ \ \ \ \ \ \ \ \ \ \ \
       p_3-p_4=(6t+3\alpha)-3(n_Y+\alpha+\theta)=3\theta-3n_3\geq \frac{27n_Y}{2}-3n_Y>10n_Y.&&&
\end{flalign*}
So we always have
\begin{flalign} \ \ \ \ \ \ \ \ \ \ \ \ \label{aeq3}
       |p_3-p_4|>10n_Y.&&&
\end{flalign}
\end{kst}

For $y_i,y_j$ in $Y\setminus Y_0$ with $i<j$,
\begin{kst}
   \item[1)]  if  $d(y_{j})=d(y_i)$  where  $|E_{4,3}(y_i)|=|E_{4,3}(y_j)|$, then  \begin{flalign*} \ \ \ \ \ \ \ \ \ \ \ \ \sigma_1(y_{j})=\sigma_1(y_{i});&&&
\end{flalign*}

   \item[2)] \vspace{2mm}  if  $d(y_{j})=d(y_i)$  where $|E_{4,3}(y_i)|-|E_{4,3}(y_j)|=4$ or $-4$, then  by (\ref{aeq3}),   \begin{flalign*} \ \ \ \ \ \ \ \ \ \ \ \ |\sigma_1(y_{j})-\sigma_1(y_{i})|=|2(p_4-p_3)|> 20n_Y;&&&
\end{flalign*}
    \item[3)] \vspace{2mm}  if  $d(y_{j})=d(y_i)$  where $|E_{4,3}(y_i)|-|E_{4,3}(y_j)|=2$ or $-2$,    then by (\ref{aeq3}), \begin{flalign*} \ \ \ \ \ \ \ \ \ \ \ |\sigma_1(y_{j})-\sigma_1(y_{i})|=|p_4-p_3|> 10n_Y;&&&
\end{flalign*}
   \item[4)] \vspace{2mm}   if  $d(y_{j})>d(y_i)$, then by (\ref{zeq2}), (\ref{aeq1}) and (\ref{aeq2}),  \begin{flalign*} \ \ \ \ \ \ \ \ \ \ \
       \sigma_1(y_{j})-\sigma_1(y_{i})&\geq \min\{p_2-p_{1},p'_2-p_{1},p_4-p_{2},p_4-p'_{2},p_4\}>3n_Y.
       &&&
\end{flalign*}
\end{kst}
 That is, any pair of partial vertex sums in $Y\setminus Y_0$ are either equal or have a difference greater than
 $3n_Y$.

 This completes the proof of Claim \ref{claim5}. \end{subproof}

 \vspace{2mm} By Claim \ref{claim5},  the final vertex sums in $Y\setminus Y_0$ will be pairwise distinct 0-sums, since  we would finally label $M\cup E_1$ with the rest $n_Y$ continuous 0-sums among
 \begin{flalign*} \ \ \ \ \ \ \ \ \ \ \ \
    O_{4}\triangleq\left\{
   \begin{array}{ll}
  \{3i|i\in[n^{\rm even}_Y+1,n^{\rm even}_Y+n_Y]\},&   {\rm when}\    n^{\rm even}_Y\ {\rm is\ odd},\   \vspace{1mm}  \\
   \{3i|i\in[n^{\rm even}_Y,n^{\rm even}_Y+n_Y]\setminus\{\frac{3n^{\rm even}_Y}{2}\}\},&   {\rm when}\    n^{\rm even}_Y\ {\rm is\ even}.  \\
   \end{array}
  \right.   &&&
\end{flalign*}

 \vspace{2mm}\noindent\textbf{Step 7.}  Label $E_1$.

\vspace{2mm} Let $O_{4,1}\subseteq O_4$    be the subset containing     the $(n_Y-n_X)$ least  labels in $O_4$.  Label $E_1$ with  $O_{4,1}$
  arbitrarily.

   In the case when  $Y_0\neq \emptyset$,     if $y'$ is covered by $M$ (let $y'x'\in M$), we ensure that the partial vertex sums of  $y'$ and $x'$ are distinct.

   If partial vertex sums of  $y'$ and $x'$ are already  distinct, no further action is required.

   If   they are equal,  we perform a switching operation to distinguish their partial sums. This adjustment preserves two key properties:   the modulo three value of every partial vertex sum in $X\cup \{y'\}$ remain unchanged; all partial vertex sums in   $Y\setminus \{y'\}$ stay the same.

  \begin{kst}
 \item[(i)]
  If $x'$ is in $X\setminus I_2$, then there exists an edge $x'{{{y''}}}$  in $E_4$. Without loss of generality, suppose $x'{{{y''}}}$ is assigned with  a $\mu$-label for some  $\mu\in\{1,2\}$. Note that, by   Lemma \ref{lem6} (ii) and   the fact that   {$|E_4(y'')|\geq 8$},
  there exists at least one  edge  $x''{y''}$ in $E_4$ other than $x'{y''}$  that is also labeled with a $\mu$-label. Swapping the labels on $x'y'', x''y''$ switches   the partial vertex sum  of each of $x',x''$ to a distinct new value (while preserving the sum modulo three), and leaves all other partial vertex sums unchanged. This achieves the desired switching. See Figure \ref{figxin1}.

    \begin{figure}[htbp]
 \centering
 %\psfrag{a}{(a)}
 %\psfrag{b}{(b)}
 \includegraphics[height=3.6cm]{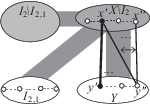}
\caption{\label{figxin1} If $x'$ is in $X\setminus I_2$,  swapping the labels on $x'y'', x''y''$ switches   the partial vertex sum  of each of $x',x''$ (while preserving the    modulo three values), and leaves all other partial vertex sums unchanged. }
\end{figure}

   \item[(ii)] \vspace{2mm}  If  $x'$ is in $I_2$, then   $x'$ is covered by $E_{F_2}\cup E_{F_3}$. Suppose
 $x'x''\in E_{F_2}\cup E_{F_3}$. Note that each edge in $E_{F_2}\cup E_{F_3}$ is labeled with a $\{1,2\}$-label.  Suppose $x'x''$ is    labeled with a $\mu$-label for some $\mu=1,2$.  Recall that there is at least one edge  $x'''{y'}$ in $E_4$ with $x'''\neq x',x''$,    that is labeled with a $\mu$-label  by  Lemma \ref{lem6} (ii) and   the fact that   {$|E_4(y'')|\geq 8$}.  Swapping the labels on edges $x'x''$ and $x'''y'$ increases the partial sums of $x'$ and $x''$ while decreasing those of $x'''$ and $y'$ when $f(x'x'') < f(x'''y')$, and conversely decreases the partial sums of $x'$ and $x''$ while increasing those of $x'''$ and $y'$ when $f(x'x'') > f(x'''y')$, all while preserving the modulo three values and leaving all other vertex partial sums unchanged, thus achieving the desired switching. See Figure \ref{figxin2}.
 \begin{figure}[htbp]
 \centering
 %\psfrag{a}{(a)}
 %\psfrag{b}{(b)}
 \includegraphics[height=3.6cm]{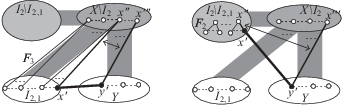}
\caption{\label{figxin2}  If  $x'$ is in $I_2$,  swapping the labels on edges $x'x''$ and $x'''y'$   distinguishes  the  partial sums of $x'$ and $y'$,     preserves the   modulo three values of partial vertex sums in $X\cup \{y'\}$, and leaves all partial vertex sums in $Y\setminus\{y'\}$ unchanged. }
\end{figure}
  \end{kst}

 \vspace{2mm}\noindent\textbf{Step 8.}  Label $M$.

 \vspace{2mm} Let $\sigma_2(v)$ be the partial vertex sum of $v$ for each $v\in V_G$ right after Step 7. Suppose $X=\{x_{i}|i\in[1,n_X]\}$ where  $\sigma_2(x_{i})\leq \sigma_2(x_{{i+1}})$ for $i\in[1,n_X-1]$. Let   $M(x_i)=\{e_i\}$ for $i\in[1,n_X]$.   Define $O_{4,2}=O_4\setminus O_{4,1}$.

  Label $M$ with   $O_{4,2}$     such that $e_i$ receives a label smaller than  $e_{i+1}$  for $i\in[1,n_X-1]$. Then the final vertex sums in $X$ are pairwise different.

  \vspace{2mm}\noindent\textbf{Step 9.}  Do some   switching when necessary.

   \vspace{2mm}  Let $\sigma(v)$ be the final vertex sum of $v$ under  $f$. Now  $\sigma(X)$ consists of pairwise distinct $\{1,2\}$-sums and  $\sigma(Y\setminus Y_0)$ consists of pairwise distinct  0-sums. If $Y_0=\emptyset$, or $Y_0\neq\emptyset$ and $\sigma(y')\notin \sigma(X)$,   then we have already obtained an antimagic labeling.

  In the following, suppose $Y_0\neq\emptyset$  and $\sigma(y')=\sigma(x_{b})$ where $\sigma(x_{b})$ is a $\mu$-label ($\mu=1,2$) for some $x_{b}$ in $X$.  Note that there exist at least one 1-sum and one 2-sum  in $(X\setminus I_2)\cup X_{2,1}$ from Step 4.1. Consider a vertex    $x_{\xi}$    in $X$  with  a $(3-\mu)$-sum, selected such that $|b-\xi|$ is minimized.  Without loss of generality, suppose $b<\xi$.

   Note that  $y'x_{b}\notin M$   from Step 7.

  \vspace{2mm}\noindent\textbf{Case 1.} If $y'x_{i}\notin M$ for any   $i\in[b+1,\xi]$.

    \vspace{2mm} In this case,   relabel $e_{i}$ with  $f(e_{{i+1}})$ for each $i\in[b,\xi-1]$ and relabel $e_\xi$ with $f(e_b)$. Let $f^\ast$ be the    new labeling.  We conclude that  \begin{itemize}
    \item For all vertices $x \in X \setminus \{x_i \mid i \in [b,\xi]\}$, $\sigma_{f^\ast}(x) = \sigma(x)$.

    \item For each $i \in [b,\xi-1]$, $\sigma_{f^\ast}(x_i) = \sigma(x_i) + 3$.

    \item The vertex $x_\xi$ satisfies $\sigma_{f^\ast}(x_\xi) = \sigma(x_\xi) - 3(\xi-b)$.

    \item The sum at $y'$ remains unchanged: $\sigma_{f^\ast}(y') = \sigma(y')$.

    \item The ordering $\sigma_{f^\ast}(x_{b-1}) < \sigma_{f^\ast}(y') < \sigma_{f^\ast}(x_i) < \sigma_{f^\ast}(x_{\xi+1})$ holds for all $i \in [b,\xi]$.

    \item The vertex $x_\xi$ is uniquely determined in $\{x_i \mid i \in [b+1,\xi]\}$ by having a $(3-\mu)$-sum.
\end{itemize} So the vertex sums in $X \cup \{y'\}$ are pairwise distinct $\{1,2\}$-sums, and those in $Y$ remain pairwise distinct $0$-sums by Claim~\ref{claim5}. Therefore, $f^\ast$ is an antimagic labeling. See Figure \ref{nx1}.
\begin{figure}[htbp]
 \centering
 %\psfrag{a}{(a)}
 %\psfrag{b}{(b)}
 \includegraphics[height=3cm]{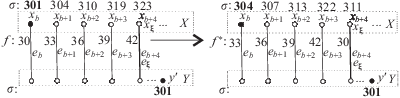}
\caption{\label{nx1} An example for the switching of the labeling in Case 1 (when  $y'x_{i}\notin M$ for any   $i\in[b+1,\xi]$).}
\end{figure}

  \vspace{2mm} \noindent\textbf{Case 2.} If  $y'x_{\lambda}\in M$ for   some  $\lambda\in[b+1,\xi-1]$.

 \vspace{2mm} In this case, $b\leq\xi-2$ since $b<\lambda$.   Relabel $e_i$ with  $f(e_{i+1})$ for each $i\in[b+1,\xi-1]$ and relabel $e_\xi$ with $f(e_{b+1})$. Let $f^\ast$ be the    new labeling. It follows that  \begin{itemize}
\item For all $x \in X \setminus \{x_i \mid i \in [b+1,\xi]\}$, $\sigma_{f^*}(x) = \sigma(x)$ holds.
\item Each $x_i$ with $i \in [b+1,\xi-1]$ satisfies $\sigma_{f^*}(x_i) = \sigma(x_i) + 3$.
\item The vertex $x_\xi$ attains $\sigma_{f^*}(x_\xi) = \sigma(x_\xi) - 3(\xi-b+1)$.
\item The modified sum $\sigma_{f^*}(y') = \sigma(y') + 3$ yields both $\sigma_{f^*}(y') < \sigma(y') + 6 \leq \sigma_{f^*}(x_{b+1})$ and $\sigma_{f^*}(y') > \sigma(y') = \sigma_{f^*}(x_b)$.
\item The ordering $\sigma_{f^*}(x_b) < \sigma_{f^*}(v) < \sigma_{f^*}(x_{\xi+1})$ holds for all $v \in \{y'\} \cup \{x_i \mid i \in [b+1,\xi]\}$.
\item The vertex $x_\xi$ is uniquely determined in $\{x_i \mid i \in [b+1,\xi]\}$ by possessing a $(3-\mu)$-sum.
\end{itemize} So the vertex sums in $X \cup \{y'\}$ are pairwise distinct $\{1,2\}$-sums, and those in $Y$ remain pairwise distinct $0$-sums, proving $f^\ast$ is antimagic. See Figure \ref{nx2}.
\begin{figure}[htbp]
 \centering
 %\psfrag{a}{(a)}
 %\psfrag{b}{(b)}
 \includegraphics[height=3cm]{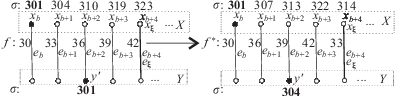}
\caption{\label{nx2} An example for the   switching of the labeling in Case 2  when  $y'x_{\lambda}\in M$ for   some  $\lambda\in[b+1,\xi-1]$.}
\end{figure}

  \vspace{2mm}\noindent \textbf{Case 3.} If    $y'x_{\xi}\in M$.

   \vspace{2mm}\noindent \textbf{Case 3.1.} If     there exists an integer   $\eta\in[1,b-1]$ such that $\sigma(x_{\eta})$ is  a $(3-\mu)$-sum.

   \vspace{2mm}  Choose $\eta$ such that $\eta$ is maximum. We relabel $e_i$ with  $f(e_{i-1})$ for each $i\in[\eta+1,b]$ and relabel $e_\eta$ with $f(e_{b})$. Let $f^\ast$ be the    new labeling.  It follows that \begin{itemize}
\item $\sigma_{f^*}(x) = \sigma(x)$ holds for all $x \in X \setminus \{x_i \mid i \in [\eta,\xi]\}$.
\item For each $i \in [\eta+1,b]$, $\sigma_{f^*}(x_i) = \sigma(x_i) - 3$.
\item The vertex $x_\eta$ satisfies $\sigma_{f^*}(x_\eta) = \sigma(x_\eta) + 3(b-\eta)$.
\item The sums at $y'$ and $x_b$ are equal: $\sigma_{f^*}(y') = \sigma(y') = \sigma(x_b)$.
\item The ordering $\sigma_{f^*}(x_{\eta-1}) < \sigma_{f^*}(x_i) < \sigma_{f^*}(y') < \sigma_{f^*}(x_{b+1})$ holds for all $i \in [\eta,b]$.
\item The vertex $x_\eta$ is uniquely determined in $\{x_i \mid i \in [\eta,b]\}$ by having a $(3-\mu)$-sum.
\end{itemize}Consequently, the vertex sums in $X \cup \{y'\}$ are pairwise distinct $\{1,2\}$-sums and those in $Y$ remain  pairwise distinct $0$-sums, making $f^\ast$ an antimagic labeling. See Figure \ref{nx3}.
\begin{figure}[htbp]
 \centering
 %\psfrag{a}{(a)}
 %\psfrag{b}{(b)}
 \includegraphics[height=3cm]{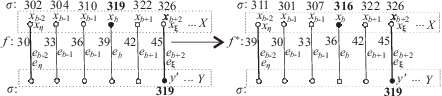}
\caption{\label{nx3} An example for the   switching of the labeling in Case 3.1 when  $y'x_{\xi}\in M$ and  there exists an integer   $\eta\in[1,b-1]$ such that $\sigma(x_{\eta})$ is  a $(3-\mu)$-sum.}
\end{figure}

   \vspace{2mm}\noindent \textbf{Case 3.2.}  If $\sigma(x_{j})$ is a $\mu$-sum for each $j \in [1, b]$.

    \vspace{2mm}   Since  $\sigma(y') = \sigma(x_{b})$ and $f(e_\xi) > f(e_b)$, we have $\sigma_2(y') < \sigma_2(x_{b})$. Let $\beta$ be the smallest integer not exceeding $b$ such that $\sigma_2(y') \leq \sigma_2(x_{\beta})$. We then relabel the edges $e_1, e_2, \ldots, e_{\beta-1}$, $e_{\xi}$, $e_{\beta}$, $e_{\beta+1}$, $\ldots, e_{\xi-1}$ in increasing order with $\{f(e_j) \mid j \in [1, \xi]\}$.
This relabeling ensures:
\begin{itemize}
\item The vertex sums in $\{x_i \mid i \in [1,\xi-1]\} \cup \{y'\}$ form pairwise distinct $\mu$-sums, all strictly smaller than $X$'s remaining $\mu$-sums; these other $\mu$-sums stay unchanged.

\item For vertex $x_\xi$, its sum becomes smaller than every other $(3-\mu)$-sum, while these $(3-\mu)$-sums maintain their original values.

\item In set $Y$, all vertex sums keep their pairwise distinct $0$-sum property.
\end{itemize} Thus, the resulting labeling is antimagic. See Figure \ref{nx4}.
\begin{figure}[htbp]
 \centering
 %\psfrag{a}{(a)}
 %\psfrag{b}{(b)}
 \includegraphics[height=3.6cm]{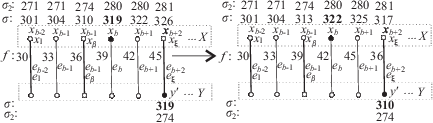}
\caption{\label{nx4} An example for the   switching of the labeling in Case 3.2 when  $y'x_{\xi}\in M$ and $\sigma(x_{j})$ is a $\mu$-sum for each $j \in [1, b]$. }
\end{figure}

  This completes the proof of Theorem \ref{thm1}.}
\end{proof}

\noindent\textbf{Funding and  Conflicts of interests}

The   author is supported by NSFC (No. 11701195), Fundamental Research Funds for the Central
Universities (No. ZQN-904) and Fujian Province University Key Laboratory of Computational Science,
School of Mathematical Sciences, Huaqiao University, Quanzhou, China (No. 2019H0015).

We wish to confirm that there are no known conflicts of interest associated with this
publication and there has been no significant financial support for this work that could have influenced its
outcome.


\begin{thebibliography}{99}
\small \setlength{\itemsep}{.8mm}

\bibitem{Diestel} R. Diestel, Graph Theory, Electronic Edition 2005, Springer, New York, 2005.

\bibitem{ref1} N. Hartsfield and G. Ringel, Pearls in Graph Theory: A Comprehensive Introduction, Academic Press INC., Boston, 1990.


\bibitem{ref2}N. Alon, G. Kaplan, A. Lev, Y. Roditty, and R. Yuster, Dense graphs are antimagic, J. Graph Theory 47   (2004), no. 4,  297-309.

    \bibitem{Eccles2016}T. Eccles, Graphs of large linear size are antimagic, J. Graph Theory 81 (2016), no. 3, 236-261.

    \bibitem{ref3}Z.B. Yilma, Antimagic properties of graphs with large maximum degree, J. Graph Theory 72   (2013), no. 4, 367-373.

\bibitem{ref17}D.W. Craston, Regular bipartite graphs are antimagic, J. Graph Theory 60 (2009), no. 3, 179-182.
\bibitem{ref16}F.-H. Chang, Y.-C. Liang, Z. Pan, and X. Zhu, Antimagic labeling of regular graphs, J. Graph Theory 82 (2016), no. 4, 339-349.

 \bibitem{ref15}K. B\'{e}rezi, A. Bern\'{a}th, and M. Vizer, Regular graphs are antimagic, Electron. J. Combin. 22 (2015), no. 3, Paper 3.34, 6 pp.

\bibitem{Deng2022} K. Deng and Y. Li, Antimagic labeling of some biregular bipartite graphs,
Discuss. Math. Graph Theory 42 (2022), no. 4, 1205-1218.

\bibitem{Yu2023} X. Yu,  Antimagic labeling of biregular bipartite graphs,
Discrete Appl. Math. 327 (2023), 47-59.
 \bibitem{Hefetz2005}  D. Hefetz, Anti-magic graphs via the combinatorial nullstellensatz, J. Graph Theory  50 (2005), no. 4, 263-272.

      \bibitem{Hefetz2010}   D. Hefetz, A. Saluz, and T.T.T. Huong, An application of the combinatorial Nullstellensatz to a graph labelling
problem, J. Graph Theory  65 (2010), no. 1, 70-82.
 \bibitem{ref6}G. Kaplan, A. Lev,  and Y. Roditty, On zero-sum partitions and anti-magic trees, Discrete Math. 309 (2009), no. 8, 2010-2014.
\bibitem{ref7}Y. Liang, T. Wong, and X. Zhu, Anti-magic labeling of trees, Discrete Math. 331 (2014), 9-14.
    \bibitem{Sierra2023} J. Sierra,  D.D.-F. Liu, and  J. Toy, Antimagic labelings of forests,
PUMP J. Undergrad. Res. 6 (2023), 268-279.
  \bibitem{Chavez2023} A. Chavez,  P. Le,  D. Lin, D.D.-F. Liu, and M. Shurman, Antimagic labeling for unions of graphs with many three-paths,
Discrete Math. 346 (2023), no. 6, Paper No. 113356, 14 pp.

\bibitem{ref4}K. Deng and Y. Li, Caterpillars with maximum degree 3 are antimagic, Discrete Math. 342 (2019), no. 6, 1799-1801.
\bibitem{ref5}A. Lozano, M. Mora, C. Seara, and J. Tey, Caterpillars are antimagic, Mediterr. J. Math. 18  (2021), no. 2, Paper No. 39, 12 pp.

    \bibitem{WDZ}C. Wu, K. Deng, and  Q. Zhao,   Antimagic labeling of subdivided caterpillars, Bull. Malays. Math. Sci. Soc. 48 (2025), no. 2, Paper No. 46, 26 pp.


\bibitem{ref18}J. Gallian, A dynamic survey of graph labeling, Electron. J. Combin.  (2024), Dynamic Survey 6,  712 pp.

  \bibitem{Konig} D. K\"{o}nig, Graphen und matrizen, Mat. Lapok. 38 (1931), 116-119.
   \bibitem{shan} S. Shan,      Antimagic orientation of graphs with
minimum degree at least 33, J. Graph Theory. 98 (2021), no. 4, 676-690.

\end{thebibliography}
\end{document}